\documentclass [11pt,a4paper,reqno]{amsart}

\usepackage{amsmath,amsfonts,amssymb,amscd,verbatim,delarray,verbatim,epigraph}

\usepackage{stmaryrd}
\usepackage{amsthm}

\usepackage[polutonikogreek,english]{babel}
\usepackage[or]{teubner}

\evensidemargin
-2mm \oddsidemargin -2mm

\DeclareMathOperator{\tr}{tr}

\chardef\bslash=`\\ 

\begin{document}


\newtheorem{Theorem}{Theorem}[section]

\newtheorem{cor}[Theorem]{Corollary}

\newtheorem{Conjecture}[Theorem]{Conjecture}

\newtheorem{Lemma}[Theorem]{Lemma}
\newtheorem{lemma}[Theorem]{Lemma}
\newtheorem{property}[Theorem]{Property}
\newtheorem{Proposition}[Theorem]{Proposition}
\newtheorem{ax}[Theorem]{Axiom}
\newtheorem{claim}[Theorem]{Claim}

\newtheorem{nTheorem}{Surjectivity Theorem}

\theoremstyle{definition}
\newtheorem{Definition}[Theorem]{Definition}
\newtheorem{problem}[Theorem]{Problem}
\newtheorem{question}[Theorem]{Question}
\newtheorem{Example}[Theorem]{Example}

\newtheorem{remark}[Theorem]{Remark}
\newtheorem{diagram}{Diagram}
\newtheorem{Remark}[Theorem]{Remark}
\newcommand{\diagref}[1]{diagram~\ref{#1}}
\newcommand{\thmref}[1]{Theorem~\ref{#1}}
\newcommand{\secref}[1]{Section~\ref{#1}}
\newcommand{\subsecref}[1]{Subsection~\ref{#1}}
\newcommand{\lemref}[1]{Lemma~\ref{#1}}
\newcommand{\corref}[1]{Corollary~\ref{#1}}
\newcommand{\exampref}[1]{Example~\ref{#1}}
\newcommand{\remarkref}[1]{Remark~\ref{#1}}
\newcommand{\corlref}[1]{Corollary~\ref{#1}}
\newcommand{\claimref}[1]{Claim~\ref{#1}}
\newcommand{\defnref}[1]{Definition~\ref{#1}}
\newcommand{\propref}[1]{Proposition~\ref{#1}}
\newcommand{\prref}[1]{Property~\ref{#1}}
\newcommand{\itemref}[1]{(\ref{#1})}


\newcommand{\CE}{\mathcal{E}}
\newcommand{\CG}{\mathcal{G}}\newcommand{\CV}{\mathcal{V}}
\newcommand{\CL}{\mathcal{L}}
\newcommand{\CM}{\mathcal{M}}
\newcommand{\A}{\mathcal{A}}
\newcommand{\CO}{\mathcal{O}}
\newcommand{\B}{\mathcal{B}}
\newcommand{\CS}{\mathcal{S}}
\newcommand{\CX}{\mathcal{X}}
\newcommand{\CY}{\mathcal{Y}}
\newcommand{\CT}{\mathcal{T}}
\newcommand{\CW}{\mathcal{W}}
\newcommand{\CJ}{\mathcal{J}}

\newcommand{\st}{\sigma}
\renewcommand{\k}{\varkappa}
\newcommand{\Frac}{\mbox{Frac}}
\newcommand{\XC}{\mathcal{X}}
\newcommand{\wt}{\widetilde}
\newcommand{\wh}{\widehat}
\newcommand{\mk}{\medskip}
\renewcommand{\sectionmark}[1]{}
\renewcommand{\Im}{\operatorname{Im}}
\renewcommand{\Re}{\operatorname{Re}}
\newcommand{\la}{\langle}
\newcommand{\ra}{\rangle}
\newcommand{\LND}{\mbox{LND}}
\newcommand{\Pic}{\mbox{Pic}}
\newcommand{\lnd}{\mbox{lnd}}
\newcommand{\GLND}{\mbox{GLND}}\newcommand{\glnd}{\mbox{glnd}}
\newcommand{\Der}{\mbox{DER}}\newcommand{\DER}{\mbox{DER}}
\renewcommand{\th}{\theta}
\newcommand{\ve}{\varepsilon}
\newcommand{\1}{^{-1}}
\newcommand{\iy}{\infty}
\newcommand{\iintl}{\iint\limits}
\newcommand{\capl}{\operatornamewithlimits{\bigcap}\limits}
\newcommand{\cupl}{\operatornamewithlimits{\bigcup}\limits}
\newcommand{\suml}{\sum\limits}
\newcommand{\ord}{\operatorname{ord}}
\newcommand{\Gal}{\operatorname{Gal}}
\newcommand{\bk}{\bigskip}
\newcommand{\fc}{\frac}
\newcommand{\g}{\gamma}
\newcommand{\be}{\beta}
\newcommand{\dl}{\delta}
\newcommand{\Dl}{\Delta}
\newcommand{\lm}{\lambda}
\newcommand{\Lm}{\Lambda}
\newcommand{\om}{\omega}
\newcommand{\ov}{\overline}
\newcommand{\vp}{\varphi}
\newcommand{\kap}{\varkappa}

\newcommand{\Vp}{\Phi}
\newcommand{\Varphi}{\Phi}
\newcommand{\BC}{\mathbb{C}}
\newcommand{\C}{\mathbb{C}}\newcommand{\BP}{\mathbb{P}}
\newcommand{\BQ}{\mathbb {Q}}
\newcommand{\BM}{\mathbb{M}}
\newcommand{\BR}{\mathbb{R}}\newcommand{\BN}{\mathbb{N}}
\newcommand{\BZ}{\mathbb{Z}}\newcommand{\BF}{\mathbb{F}}
\newcommand{\BA}{\mathbb {A}}
\renewcommand{\Im}{\operatorname{Im}}
\newcommand{\idd}{\operatorname{id}}
\newcommand{\ep}{\epsilon}
\newcommand{\tp}{\tilde\partial}
\newcommand{\doe}{\overset{\text{def}}{=}}
\newcommand{\supp} {\operatorname{supp}}
\newcommand{\loc} {\operatorname{loc}}
\newcommand{\de}{\partial}
\newcommand{\z}{\zeta}
\renewcommand{\a}{\alpha}
\newcommand{\G}{\Gamma}
\newcommand{\der}{\mbox{DER}}

\newcommand{\Spec}{\operatorname{Spec}}
\newcommand{\Sym}{\operatorname{Sym}}
\newcommand{\Aut}{\operatorname{Aut}}

\newcommand{\Idd}{\operatorname{Id}}

\newcommand{\tG}{\widetilde G}

\newcommand{\FX}{\mathfrac {X}}
\newcommand{\FV}{\mathfrac {V}}
\newcommand{\SX}{\mathcal {X}}
\newcommand{\SV}{\mathcal {V}}
\newcommand{\SO}{\mathcal {O}}
\newcommand{\SD}{\mathcal {D}}
\newcommand{\Sr}{\rho}
\newcommand{\SR}{\mathcal {R}}



\title {Geometry and arithmetic of verbal dynamical systems on simple groups}

\author[Bandman, Grunewald, Kunyavski\u\i , Jones] {Tatiana Bandman, Fritz Grunewald,
and Boris Kunyavski\u\i \\ with an appendix by Nathan Jones}
\address{Bandman: Department of
Mathematics, Bar-Ilan University, 52900 Ramat Gan, ISRAEL}
\email{bandman@macs.biu.ac.il}

\address{Grunewald: Mathematisches Institut der
Heinrich-Heine-Universit\"at D\"usseldorf, Universit\"atsstr. 1,
40225 D\"usseldorf, GERMANY}
\email{grunewald@math.uni-duesseldorf.de}

\address{Kunyavskii: Department of
Mathematics, Bar-Ilan University, 52900 Ramat Gan, ISRAEL}
\email{kunyav@macs.biu.ac.il}

\address{Jones: Department of Mathematics, 
University of Mississippi, 
Hume Hall 305,
P. O. Box 1848,
University, MS  38677-1848, 
USA}
\email{ncjones@olemiss.edu}

\begin{abstract}

We study dynamical systems arising from word maps on simple groups. 
We develop a geometric method based on the classical trace map for 
investigating periodic points of such systems. These results lead 
to a new approach  to the search of Engel-like sequences of
words in two variables which characterize finite solvable groups. 
They also give rise to some new
phenomena and concepts in the arithmetic of dynamical systems.

\end{abstract}

\maketitle


\begin{epigraph}
{\textgreek{\s{e}n \s{a}rq\hci \ \Cs{h}n \r{o} l\'ogos\dots}}
{\textgreek{KATA IWANNHN 1:1}\footnotemark[1]}

\end{epigraph}

\footnotetext[1]{{In the beginning was the Word\dots}{John 1:1}}

\begin{quote}{\fontsize{8pt}{3.7mm}
\selectfont\tableofcontents}
\end{quote}



\section{Introduction} \label{intro}

The initial goal of the present paper was to get deeper understanding of what is behind recent 
results achieved in describing the class of finite solvable groups by identities in two 
variables \cite{BGGKPP1}, \cite{BGGKPP2}, \cite{BWW}. Although the results were purely group-theoretic, 
it was clear that the key role is played by geometry and dynamics. Byproducts of this investigation seem 
to us not less interesting than the initial problem. 

We reformulated the original problem in the language of a verbal dynamical system on an algebraic group $G$ 
(the notion of its own interest). We study these systems for the case $G=SL(2)$, the most important for the 
initial group-theoretic problem. Towards this end, we 
\begin{itemize}
\item prove several surjectivity theorems for the classical trace map over finite fields; 
\item introduce a new method based on the trace map and these theorems. 
\end{itemize}

This allowed us not only to explain the mechanism of the proofs from the above cited papers but to obtain a method for 
producing more sequences of the same nature.

These arithmetic-geometric considerations led us to a new notion of residual periodicity of a dynamical system which reflects 
its local-global behaviour. This concept will hopefully yield new results in the arithmetic of dynamical systems on algebraic 
varieties. Here we present some primary examples and propose some conjectures.


To be more  precise, let $F_{r+s}$ be the free group with basis $x_1,\dots ,x_s, u_1,\dots ,u_r,$ and let

\begin{equation}\label{i1}
\CW=\left\{\begin{aligned}{} & w_1(x_1,\dots ,x_s, u_1,\dots ,u_r),\\
& \dots ,\\
& w_r(x_1,\dots ,x_s, u_1,\dots ,u_r).\end{aligned}
\right\}\end{equation}
be an $r$-tuple of words in $F_{r+s}.$  Thus for any group $G$ we obtain a self-map:
\begin{equation} 
D_{\CW}:G^{r+s}\to G^{r+s},
\label{selfmap}
\end{equation}
$$(g_1,\dots, g_s, v_1,\dots, v_r)\mapsto (g_1,\dots, g_s, w_1(g_1, \dots,g_s, v_1,\dots,v_r),\dots, w_r(g_1, \dots,g_s, v_1, \dots,v_r).$$

Choosing $G$ to be a linear algebraic group defined over some field $k$, we thus find 
a polynomial self-map of the underlying affine variety $G^{r+s}$ attached to $\CW.$

 A set $M\subset G^{r+s}$ is called invariant if $D_{\CW}(M)\subset M.$

 For our purposes it is important to introduce initial conditions
 and, for every group $G,$ a so-called forbidden set.
 Let $\CJ=(f_1(x_1,\dots ,x_s),\dots ,f_r(x_1,\dots ,x_s))$ be words in $F_s.$
 Then given $G$ and $(g_1,\dots, g_s)\in G^s$  we have an iterative sequence of
 $r$-tuples of elements of $G$:
 $$e_0=(f_1(g_1,\dots ,g_s),\dots ,f_r(g_1,\dots ,g_s)),\dots ,$$
 $$e_{n+1}=(w_1(g_1,\dots, g_s, e_n),\dots , w_r(g_1,\dots, g_s, e_n)),\dots $$ 
 We are interested in finding $(g_1,\dots ,g_s)$ such that the sequence 
 $e_0, e_1,\dots$ has certain properties. To find such $(g_1,\dots ,g_s)$,  
 we add $s$  extra ``tautological'' variables 
and obtain a self-map as in \eqref{selfmap}.

Then given $\CW$, $G$ and $\CJ$, we have an iterative sequence:

$$e'_0=(g_1,\dots , g_s, f_1(g_1,\dots ,g_s),\dots ,f_r(g_1,\dots ,g_s)),\dots ,$$
$$e'_{n+1}=D_{\CW}(e'_n),\dots $$

The forbidden set consists of the choice of an invariant set $I_G\subset G^{r+s}$
for every group $G.$

We call the triple $ D= (\CW,\CJ, I_G)$ a {\it verbal dynamical system}. We are interested
in invariant sets disjoint from  $I_G.$



\begin{Remark} \label{restr}
It is sometimes convenient to modify this general setup as follows.
 
(i) It may happen that the $r$-tuple $\CW$ depends on less than $r+s$ variables 
(say, of $x_1,\dots ,x_s$ only $x_1,\dots ,x_t$, $t<s$, show up in $\CW$ whereas 
the rest of the $x_i$ only appear in the initial conditions $\CJ$). In such a case, 
we will restrict our dynamical system to $G^{r+t}$  (in particular, the forbidden set 
is also chosen inside $G^{r+t}$). See Example \ref{s=r=1} below. 

(ii) One can fix an $s$-tuple $(g^{\circ }:=(g^{\circ }_1,\dots ,g^{\circ }_s)\in G^s$ and 
consider the corresponding ``fibre'' of our dynamical system $D^0_{\CW}\colon G^r\to G^r$ defined by 
$$D^0_{\CW}((v_1,\dots, v_r))=(w_1(g^{\circ }_1, \dots,g^{\circ }_s, v_1,\dots,v_r),\dots, w_r(g^{\circ }_1, \dots, g^{\circ }_s, v_1, \dots,v_r)).$$
In particular, for $r=1$ we arrive at a self-map $G\to G$. This simplified system will be largely used in what follows. 
\end{Remark}

\begin{Example} \label{s=2,r=1}
Take $s=2$, $r=1$ and  consider a triple $D_1$ consisting of

$$\CW=([xux^{-1},yuy^{-1}]),$$
$$\CJ=(x^{-2}y^{-1}x),$$
$$ I_G=\{G\times G\times\{1\}\}.$$

The corresponding map  is $$D_{\CW}(x,y,u)=(x,y,[xux^{-1},yuy^{-1}]).$$
The associated iterative sequence is
$$e_0=x^{-2}y^{-1}x,\quad 
e_1=[x^{-1}y^{-1},yx^{-2}y^{-1}xy^{-1}],\quad e_2=[xe_1x^{-1},ye_1y^{-1}],\dots $$

A key step in our characterization of finite solvable
groups \cite{BGGKPP1}, \cite{BGGKPP2} can now be reformulated as follows:
\end{Example}

\begin{Theorem} \label{ours}
For $G=SL(2,q)$  the dynamical system  $D_1$ has a fixed point outside $I_G$
for every $q>3$.
\end{Theorem}

A key step to the characterization obtained in \cite{BWW} can be reformulated 
in a similar way:

\begin{Example} \label{s=r=1}
Take $s=2$, $r=1$, $\CW=([y^{-1}uy, u^{-1}])$, $\CJ=(x).$
As the variable $x$ does not show up in $\CW$ but only appears in $\CJ$ (and so $t=1$), we proceed 
as in Remark \ref{restr}(i) and consider the restricted system 
$G^2\to G^2$, $(y,u)\mapsto (y, [y^{-1}uy, u^{-1}])$, 
with the forbidden set $I_G:=\{G\times \{1\}\}$. Denote this system by $D_2$. 

The associated iterative sequence is
$$e_0=x,\quad 
e_1=[y^{-1}xy,x^{-1}],\quad e_2=[y^{-1}e_1y, e_1^{-1}],\dots $$

The main result of \cite{BWW} can now be read off as follows:
\end{Example}

\begin{Theorem} \label{Wilsons}
For $G=SL(2,q)$  the dynamical system  $D_2$ has a periodic point outside $I_G$
for every $q>3$.
\end{Theorem}

In the present paper we mostly restrict ourselves to considering 
the most important case $G=SL(2,k)$  (though in Section \ref{sec:Suzuki}
we also consider the  Suzuki groups).

In the case $G=SL(2,k)$ we introduce a new method based on classical
results of  Klein, Fricke, Vogt,   
Magnus 
from which   
it follows (see, e.g., \cite{Pe2}) that there is a polynomial map $\psi\colon\BA^N(k)\to \BA^N(k)$
making the diagram

\begin{equation}
\begin{CD}
G^{s+r} @>{D_{\CW}}>> G^{s+r} \\
@V\pi VV   @V\pi VV\\
\BA^N(k) @>{\psi}>> \BA^N(k)\label{i2}
\end{CD}
\end{equation}

 commutative.  Here $\pi$ is defined using the traces of products as in \thmref{Hor} 
 below.

 In the case $r=1$, $t=1$ the projection $\pi\colon SL(2,k)^2\to \BA^3(k)$ is defined as
 $$\pi(x,y)=(\tr(x), \tr(xy), \tr(y)).$$

 In the case $r=1$, $s=2$ the map $\pi\colon SL(2,k)^3\to \BA^7(k)$ is defined as
 $$\pi(x,y,u)=(\tr(x), \tr(y), \tr(u), \tr(xy), \tr(xu), \tr(yu), \tr(xyu)),$$
 and the image of $\pi$ is contained in a hypersurface $Z\subset\BA^7$ (see 
 \eqref{hha1} below for an explicit equation of $Z$).

 We prove the following surjectivity theorems
 (see Theorems \ref{rational} and \ref{surj} below).

 \begin{nTheorem}\label{surjectivity1}
 For any point $ \ov{a}=(s_0,u_0,t_0)\in \BA^3(\BF_q)$ the set $\pi^{-1}(\ov{a})\subset SL(2,q)^2$ is nonempty. 
 \end{nTheorem}

\begin{nTheorem}\label{surjectivity2}
For any point $y\in Z(\BF_q)$ the set $\pi^{-1}(y)\subset SL(2,q)^3$
is nonempty.
\end{nTheorem}

These surjectivity theorems allow us to obtain sufficient conditions for the existence of fixed points 
of the reduced (modulo $p$) dynamical system, uniformly in $p$, and treat concrete examples arising from 
\cite{BGGKPP1}, \cite{BGGKPP2}, \cite{BWW}. 

\medskip

On the other hand, the above dynamical reinterpretation of our 
group-theoretic problem leads to some interesting
``local-global" properties of dynamical systems on algebraic varieties.
By an {\it AG dynamical system} (AG stands for arithmetic-geometric) we mean a triple $ D=(X,V,\vp),$ where
\begin{itemize}
\item either $X$ is an algebraic variety  defined  over a global field $K, $
$\varphi\colon X\to X$ is a dominant endomorphism and  $V\subset
X(K)$ is a subset invariant under $\vp;$
\item or $X$ is an $\SO$-scheme ($\SO$ stands for the ring of integers in
$K$), $\varphi\colon X\to X$ is dominant and $V\subset X(\SO)$ is a
$\vp$-invariant subset.
\end{itemize}

A periodic point is a fixed point of an iteration $\varphi^{(n)}$ of
$\vp.$ Together with the system $D=(X,V,\vp)$, we consider its
reductions  $D_p=(X_p,V_p,\vp_p),$ where $p$ ranges over all but
finitely many places of $K$ (see \secref{final} for precise
definitions). For each reduction, we consider the length $\ell_p$ of
the shortest orbit $C_p$ which does not intersect the ``forbidden''
set $V_p\subset X_p$. If such an orbit does not exist, we set
$\ell_p=\infty.$ We are interested in the distribution of
$\ell_p$'s. More specifically, let $M\subset\mathbb N$ be the set of
all primes $p$ such that $\ell_p=\infty.$  Let $N=\{\ell_p: p\not\in
M\}.$

\begin{itemize}
\item\label{tt1} If $M$ is  infinite, we call the system {\bf  residually aperiodic}.

\item \label{tt2} If $M$ is  finite, we call the system
{\bf residually periodic}.

\item \label{tt3} If both $M$ and  $N$ are  finite, we call the
system {\bf strongly  residually periodic}.
\end{itemize}

\noindent Precise definitions, examples and discussion of these
notions are the subject of Section \ref{final}.

\begin{Remark}\label{Hr}
According to a theorem of Hrushovski \cite{Hr}, $\vp$ has a periodic
point in $X(\ov{\mathbb F}_p)\setminus V(\ov{\mathbb F}_p)$ provided
$X$ is an affine $\mathbb F_p$-variety and $V$ is a proper affine subset of $X$ 
($\ov{\mathbb F}_p$ stands for the algebraic closure of $\mathbb
F_p$). In contrast, we are only interested in periodic  points in
$X(\mathbb F_p)$.
\end{Remark}


In this language our approach to the problem of characterization 
of finite solvable groups looks as follows.  
We consider word maps of groups $G=SL(2,q)$. 
For every word map $\varphi\colon G^m\to G, \quad m=2,3$ 
(and an additional  word $f\colon G^2\to G$ in the case $m=3$) 
we define a verbal dynamical system (see, e.g., Examples \ref{s=2,r=1}, \ref{s=r=1}). 
Regarding the group as an affine variety, we obtain from a verbal dynamical system
an AG dynamical system on an affine $\BZ$-scheme. 
(In Example \ref{s=2,r=1} we have $X=SL(2)\times SL(2)\times SL(2)$, 
$V=SL(2)\times SL(2)\times \{1\}$, $\vp(x,y,u)= (x,y,[xux^{-1},yuy^{-1}])$, in Example 
\ref{s=r=1} we have $X=SL(2)\times SL(2)$,  
$V=SL(2)\times \{1\}$, $\vp(y,u)= (y,[y^{-1}uy,u^{-1}])$.) 
The word map is a ``good" candidate if and only if that system is residually
periodic. 
Using the trace map we simplify the AG system by including it into a
commutative diagram

\begin{equation}
\begin{CD}
X  @>{\tilde\vp}>> X   \\
@V\pi VV   @V\pi VV\\
Y @>{\psi}>> Y\label{ddd2}
\end{CD}
\end{equation}

\noindent where  $\pi$ is a surjective projection, defined over
$\BZ,$ and $\psi$ is the trace map  (see Subsections
\ref{2vm}, \ref{3vm} for more details). Moreover, the dynamical system
$D'=(Y,\pi(V),\psi)$ has special geometric properties allowing us to
find out when it is strongly residually periodic. Note that $\pi$ is
surjective, therefore if $D'$ is strongly  residually periodic then
$D$ is residually periodic.



\medskip

It is an interesting question what arithmetic or geometric
conditions can guarantee residual periodicity (or aperiodicity) of a
given dynamical system.  Certainly, if the forbidden set $V$ is
empty then the system is residually periodic.

The role of arithmetic may be demonstrated by the following
example.

\begin{Example} \label{splitfield}
Let $a$ and $b$ denote distinct integers, and let
$H(x)=(x^2-a)(x^2-b)(x^2-ab)+x.$ The polynomial $H(x)$ defines a
morphism $H\colon\BA_{\BZ}^1\to\BA_{\BZ}^1.$

For every $p$ the reduced morphism $H_p$ has fixed points. Indeed,
if $p\vert a$ or $p\vert b$, we have $H_p(0)=0$. If none of $a$ and
$b$ is divisible by $p$, we can use the fact that the Legendre
symbol is a multiplicative function and conclude that at least one
of three numbers: $a$, $b$, $ab$, is a square modulo $p$. A square
root of this number is then a fixed point of $H_p$, so we have
$\ell_p=1.$

On the other hand, the  morphism  $H\colon\BA_{\BZ}^1\to\BA_{\BZ}^1$
may have no periodic points. Indeed, according to \cite{Na}, the
period of a rational point for a monic polynomial cannot exceed 2,
and Magma computations show that  for $a=2$, $b=3$
there is no rational solution to the equation $H(H(x))-x=0.$
\end{Example}

This example shows that one of the reasons for residual periodicity
may be the existence of periodic points defined over a splitting
field. Polynomials of that kind were studied in \cite{BB},
\cite{Br}, \cite{BBH}, \cite{So}.

\medskip

As to geometric conditions, the dynamical system under consideration 
may happen to be residually periodic because of the existence of 
invariant functions (say, when there is an ``extra'' coordinate on 
which $\varphi$ acts trivially) as in the following simple example.

\begin{Example} \label{fibered}
Let
$D=(X,V,\vp),$ where $X=\BA^2,$ $V= \{(a,b)\in X: a=\pm 1  \text{ or
} b=\pm 1\text { or }a=0\text{ or }b=0\}$, and $\vp(a,b)=(a^2b,b).$
Consider the integral model $\SD=(\SX,\SV,\Phi)$ where
$\SX=A_{\BZ}^2, \ \  \SV=\{(a,b)\in X(\BZ): a=\pm 1  \text{ or }
b=\pm 1\text { or }a=0\text{ or }b=0\}$ and  $\Phi(a,b)=(a^2b,b).$
We have  $M=\{2,3\}$. 
The variety of fixed
points of $\Phi$ is a curve $C=\{(a,b) : ab=1\}$, $C\bigcap
\SV=\{\pm(1,1)\}.$ Nevertheless, for any prime $p>3$ we have
$C_p\setminus V_p\neq\emptyset,$  i.e. $\ell_p=1.$
\end{Example}

These examples show that there are at least two general reasons for a
dynamical system to be  strongly residually periodic. The first one
is purely arithmetic as in  \exampref{splitfield}. Our first
observations show that even in the simplest cases of one-dimensional
systems, arising questions are related to nontrivial arithmetical
problems. In the case of elliptic curves, one of such problems has been 
solved by N.~Jones by establishing a weakened version of the 
long-standing Koblitz's conjecture (see the appendix to the present 
paper).

The second one is of  geometric nature as for the trace map above.
This map has an invariant function which leads to the dimension
jump for the variety of fixed points.
Once we can prove that this variety $W$ is absolutely irreducible
(or at least contains an absolutely irreducible component), we can
apply the Lang--Weil estimates \cite{LW} to conclude that there
exists a fixed point on the reduction $W_q$ for $q$ big enough.
(Of course, if $\dim W=1$, classical Weil's estimates (see, e.g.,
\cite{FJ}) are quite enough.)

We believe that residually periodic dynamical systems is an object
worthy of investigation. The following particular case seems to be
especially interesting. Consider a  $\BZ$-scheme $X,$ a dominant
endomorphism $\vp $ of $X$, and define $V$ as  the union of all
finite $\vp $-orbits in $X(\BZ).$  Then $V_p$ is the union of orbits
of the reductions of all preperiodic points of $\varphi$. In simple
words, this means that in this case we are interested in the
distribution of the smallest periods of the  maps $\varphi_p$ not
coming from preperiodic points of $\varphi$. To the best of our
knowledge, such a classification of dynamical systems according to
their ``hidden'' periodicity did not appear in the literature.

\medskip

The structure of the paper is as follows.

Section \ref{psl} contains a general framework of our method for the
most important case $G=PSL(2,q)$. The Suzuki groups are treated in
Section \ref{sec:Suzuki}. Applications to concrete sequences are
contained in Section \ref{Examples}. Section \ref{final} is
completely devoted to the new notion of residually periodic
dynamical systems. We give basic definitions, consider simple
examples and state some conjectures. 
The appendix contains a theorem of N.~Jones answering one 
the questions posed in Section \ref{final}.

\section{Notation and preliminaries}\label{np}

Recall that in \cite{BGGKPP1}, \cite{BGGKPP2},
\cite{BWW} there have been exhibited explicit families
$\alpha_n(x,y)$, $\beta_n(x,y)$ of words in 
$F_2$ allowing one to characterize the class $\mathcal S$
of finite solvable groups in the class of all finite groups as
follows:

\medskip

\noindent {\it A finite group $G$ belongs to $\mathcal S$ if and
only if there exists $n$ such that $G$ satisfies the identity
$\gamma_n(x,y):=[\alpha_n(x,y),\beta_n(x,y)]\equiv 1.$ }

\medskip

Here $[a,b]=aba^{-1}b^{-1}$ denotes the commutator.

As in the introduction, we produce these recurrence formulas using the 
dynamical viewpoint. We consider the dynamical systems $D_1$ and $D_2$ 
from Examples \ref{s=2,r=1} and \ref{s=r=1}, respectively, 
and consider their fibres as in Remark \ref{restr}(ii). 
This means that for any group $G$ we introduce the 
maps $G\to G$: $\rho_{u,v}(w):=[uwu^{-1}, vwv^{-1}]$,
$\sigma_u(w):=[u^{-1}wu,w^{-1}]$. Then the $n$-th term of the characterizing sequence 
can be written as the $n$-th iteration of the map $\rho$ (resp. $\sigma$):
\begin{equation}
\gamma_n(x,y)=\rho^{(n)}_{x,y}(\gamma_0(x,y)) \label{eq:ours-it}
\end{equation}
(resp.
\begin{equation}
\gamma_n(x,y)=\sigma^{(n)}_{y}(\gamma_0(x,y))), \label{eq:BWW-it}
\end{equation}
where $\gamma_0(x,y)=x^{-2}y^{-1}x$ (resp. $\gamma_0(x,y)=x$).

Suppose that $S$ is a solvable group of derived length $n$. Then the
recursive structure of the above formulas shows that
$\gamma_n(x,y)\equiv 1$ in $S$. To establish the converse statement,
it is enough to show that the identity $\gamma_n(x,y)\equiv 1$ does
not hold in any finite minimal simple non-solvable group $G$. (That
is precisely what was done in \cite{BGGKPP1}, \cite{BGGKPP2},
\cite{BWW}.)

To establish this fact in the case of sequences of type
\eqref{eq:BWW-it}, it is enough to show that there exists $u=y_0\in G$
such that the map $\sigma_u$ has a (non-identity) {\bf periodic}
point, i.e. there exist a positive integer $m$ and an element $1\ne
g\in G$ such that $g$ can be written in the form $g=\gamma (x,y_0)$
and $\sigma_{y_0}^{(m)}(g)=\sigma_{y_0}(g)$. (For sequence
\eqref{eq:BWW-it}, that is precisely what was done in \cite{BWW}.) It
is important to note here that every point has a finite orbit (i.e.
is preperiodic in the sense of \cite{Si1}) but {\it a priori} it can
happen that all these  orbits contain identity, which being fixed is
the only periodic point. We need an orbit that never hits the
identity, and therefore contains another periodic point. This explains our choice of 
the forbidden set in Examples \ref{s=2,r=1} and \ref{s=r=1}.

Let us recall the list of minimal simple non-solvable groups \cite{Th}:

{\rm (1)} $G=PSL(2,p)$, $p=5$ or $p\equiv \pm 2 \pmod 5$, $p\ne 3$,

{\rm (2)} $G=PSL(2, 2^p)$,

{\rm (3)} $G=PSL(2, 3^p)$, $p$ is an odd prime,

{\rm (4)} $G=Sz(2^p)$, $p$ is an odd prime,

{\rm (5)} $G=PSL(3,3)$.

Here $Sz$ stands for the Suzuki group (twisted form of $B_2$, see,
e.g., \cite{HB} for details).

To obtain a characterization of finite solvable groups, we wish to
find a word $\varphi\in F_2(x,y)$ with the following properties:

(i) for any finite solvable group $S$ there exists an integer $n$
such that for every $y\in S$ the map $\varphi_y^{(n)}\colon S\to S$
is the identity map (here $\varphi_y(x):=\varphi (x,y))$; 

(ii) for each finite simple non-solvable group $G$ from the above
list, there exists $y\in G$ such that the self-map $\varphi_y\colon
G\to G$ has a non-identity periodic point. For the $PSL(2)$ case,
this fits into the approach described in \secref{intro}: we consider
the dynamical system $(PSL(2,\BZ),\{1\},\varphi_y)$ and all its
reductions. (Note that in our context, the difference between $SL$ and $PSL$ 
is negligible, see Remark  \ref{SL-PSL} below.)

In order to satisfy condition (i), one has to impose some
restrictions on $\varphi$. We shall discuss this matter in Section
\ref{final}.

In the sequel, we shall consider two separate cases: $G=PSL(2,q)$
and $G=Sz(q)$ (the case of the single group $G=PSL(3,3)$ is usually
easy to handle). In each case we will show that the corresponding
dynamical system $D$ gives rise to a dynamical system $D'$ in the
space of traces (the  trace map) as in diagram \eqref{ddd2}. The
trace map has special geometry: the set of its fixed points (or of
periodic points of bounded period) has positive dimension. This
allows us to formulate a geometric sufficient condition on $\vp$ in
order to get a periodic point in every reduction. (See
\secref{final} where we dare formulate some general conjectures.)

Further on we denote by $\BA_{x_1,\dots ,x_n}^n$ the affine space with
coordinates $x_1,\dots ,x_n.$  

For brevity, we denote $\tG = SL(2,q)$.

We will repeatedly use expressions of the form ``a rational curve with 
$n$ punctures'' (even if our curve lies in an affine space) referring 
to an open subset of a projective curve of genus zero whose complement 
consists of $n$ points (e.g., the curve $xy=1$ in the affine plane will 
be referred to as a rational curve with two punctures).

\section {Case $G=PSL(2,q)$}\label{psl}

In this section we show how every word map gives rise to a dynamical
system. Then we  prove that this dynamical system may be included
into a commutative diagram of type \eqref{ddd2} (namely, diagrams
\eqref{d4} and \eqref{d6} below). The idea is that it is sufficient to
look for periodic points of the trace map $\psi.$ Indeed, if a point
$a$ is $\psi$-periodic, then all the points in the fibre over $a$
are $\vp$-periodic. The problem is to show that this fibre is not
empty. We first show how to construct the trace map, then we show
that the projection is a surjective morphism for every reduction
(Theorems \ref{rational} and \ref{surj}). Specific geometry of the  trace
map allows us to give sufficient conditions for the correewponding dynamical 
system to be residually periodic (Theorems \ref{condition} and \ref{s3vm}).

\medskip

Our method is based on the following classical fact (\cite{Vo},
\cite{Fr}, \cite{FK}, \cite{Ma1}) cited here from the paper \cite{Ho} (see also
\cite{Ma2}, \cite{Go} for a nice modern exposition of these results).

\begin{Theorem}\label{Hor}
Let $F=\left<a_1, \dots, a_n\right>$ denote the free group on $n$
generators. Let us embed $F$ into $SL(2,\BZ)$ and denote by $\tr$
the trace character. If $u$ is an arbitrary element of $F$, then the
character of $u$ can be expressed as a polynomial
$$
\tr(u)=P(t_1, \dots, t_n, t_{12},\dots ,t_{12\dots n})
$$
with integer coefficients in the $2^{n}-1$ characters
$t_{i_{1}i_{2}\dots i_{\nu}}=\tr(a_{i_{1}}a_{i_{2}}\dots
a_{i_{\nu}})$, $1\le\nu\le n, $ $1\le i_{1}<i_{2}< \dots <i_{\nu}\le
n.$ \qed
\end{Theorem}

Note that the theorem remains true for the group $\tG =SL(2,q)$
(and, more generally, for $SL(2,R)$ where $R$ is any commutative
ring, see \cite{CMS}).

We shall use this theorem in two different situations: for maps
arising from formulas of type \eqref{eq:BWW-it}, called two-variable
maps, and for those arising from formulas of type
\eqref{eq:ours-it}, called three-variable maps. These situations
will be described in the next two subsections respectively.

\subsection{Two-variable maps} \label{2vm}
In this section  we 
focus on the underlying affine algebraic variety of the algebraic group $\tG$. 
Consider a  morphism  $\vp\colon \tG\times \tG\to \tG$ satisfying the
property (needed for descending to $G=PSL(2)$): 
$$
\vp(\pm x,\pm y)=\pm\vp(x,y).
$$
For example, any word map provides such a morphism. Namely, 
for any $x, y\in \tG$ denote $s=\tr(x)$, $t=\tr(y)$, and
$u=\tr(xy)$, and define a morphism $\pi\colon \tG\times \tG\to
\BA_{s,u,t}^3$ by
$$\pi(x,y):=(s,u,t).$$
Then in view of \thmref{Hor} there exists a map $\psi \colon
\BA_{s,u,t}^3\to \BA_{s,u,t}^3$ such that
\begin{equation}\label{cc}
\psi(\pi(x,y))=\pi(\vp(x,y),y).
\end{equation}
This map is called  a ``trace map" and is widely used (see, e.g., 
\cite{Pe2}).

Define $\tilde\vp =(\vp,  id)\colon \tG\times \tG\to \tG\times \tG$
by $\tilde\vp(x,y)=(\vp(x,y),y)$. Then the following diagram
commutes:
\begin{equation}
\begin{CD}
\tG \times \tG  @>{\tilde\vp}>> \tG \times \tG   \\
@V\pi VV   @V\pi VV\\
\BA^3_{s,u,t} @>{\psi}>> \BA^3_{s,u,t} \label{d4}
\end{CD}
\end{equation}

Here $\psi(s,u,t):=(f_1(s,u,t),f_2(s,u,t),t)$, where
$f_1(s,u,t)=\tr(\vp(x,y)),$ $f_2(s,u,t)=\tr(\vp(x,y)y).$


\begin{Lemma}\label{positive}
For any word map $\vp(x,y)$
 the variety
$$\Phi : \{f_1(s,u,t)=s, f_2(s,u,t)=u\}\subset\BA^3_{s,u,t}$$
of fixed points of $\psi$ has positive dimension.\end{Lemma}

\begin{proof}
Since the variety $\Phi $ is defined
by two equations in $\BA^3_{s,u,t},$
it is sufficient to show that it is not empty.
But for any word $\om(x,y)$
we have: $\om(1,1)=1,$ thus $\psi(2,2,2)=(2,2,2),$
hence $\Phi\neq\emptyset.$\end{proof}

\bigskip

\begin{Lemma}\label{fix}
Let $Q=(s_0, u_0,t_0)$ be a fixed point of $\psi$ defined over
$\BF_q$. Let $(x,y)\in\pi^{-1}(Q)$. Then
$(\vp(x,y),y)\in\pi^{-1}(Q)$ as well.
\end{Lemma}

\begin{proof}
Indeed, \eqref{cc} gives $ \pi (\vp(x,y),y)=\psi(Q)=Q.$
\end{proof}

\begin{Theorem}\label{rational}
For every $\BF_q$-rational
point $Q=(s_0, u_0,t_0)\in \BA_{s,u,t}^3$ the fibre   
$H=\pi^{-1}(Q)$ has an $\BF_q$-rational point. 
\end{Theorem}

\begin{proof}
We will look for an element of 
$H$ among pairs of matrices of the form 
\begin{equation}  
\left(
\begin{pmatrix}
0 & 1 \\
-1 & s_0
\end{pmatrix} , 
\begin{pmatrix}
a & b \\
c & -a+t_0
\end{pmatrix}
\right).
\label{pairs}
\end{equation}
To lie in $H$, the entries of these matrices must satisfy the equations
$$
a(-a+t_0)-bc=1, \quad c-b+s_0(-a+t_0)=u_0.
$$
On eliminating $b$, we arrive at the following equation in $a$ and $c$: 
\begin{equation} 
a^2+c^2-s_0ac-t_0a +(s_0t_0-u_0)c+1=0, 
\label{ac}
\end{equation}
which has a solution for every $q$. Of course, this can be proved using 
the Chevalley--Warning theorem, but for the reader's convenience we 
present here an elementary proof. 

{\bf Case 1. $q$ is odd.} 

The discriminant $D$ of the quadratic part of the left-hand side of \eqref{ac} 
equals $s^2_0-4$. If $D=0$, i.e. $s_0=\pm 2$, we exhibit an explicit point 
in $H$: 
\begin{equation}
\left( \begin{pmatrix}\pm 1 & u_0\mp t_0\\0 &\pm 1\end{pmatrix}, 
\begin{pmatrix}1 & t_0-2\\ 1 & t_0-1\end{pmatrix}\right),
\label{D=0}
\end{equation}
so we may assume $D\ne 0$. 
First, by a linear change of variables over $\BF_q$, let us bring \eqref{ac} to the form 
$$
\tilde{a}^2+\varepsilon\tilde{c}^2=r.
$$ 
If $r$ is a square, $r=v^2$, we can put $\tilde a=v$, $\tilde c=0$, 
so we may assume that $r$ is not a square. If $\varepsilon$ is not a square, then 
$r/\varepsilon$ is a square, $r/\varepsilon=v^2$, and we can put $\tilde a=0$, $\tilde c=v$, 
so we may assume $\varepsilon$ is a square, $\varepsilon =v^2$. In $\BF_q$ there are 
$(q+1)/2$ squares and $(q-1)/2$ nonsquares, thus among $(q+1)/2$ elements $r-\tilde a^2$, when $\tilde a$ ranges over 
$\BF_q$, there is a square $w^2$. We then put $\tilde c=w/v$. 

{\bf Case 2. $q$ is even}. 

If $s_0=0$, then we get a point in $H$ from \eqref{D=0}, so we may assume $s_0\ne 0$. 
Then on putting $\tilde a=a+(s_0t_0+u_0)/s_0$, $\tilde c=c+t_0/s_0$, we bring \eqref{ac} to the form 
$$\tilde a^2+\tilde c^2+s_0\tilde a\tilde c=r.$$
As every element of $\BF_q$ is a square, we have $r=v^2$ and we can put $\tilde a=v$, $\tilde c=0$. 
\end{proof}

 \begin{cor}\label{con}  Consider the following ``conjugation"
equivalence relation $\sim$ on $SL(2,\BF_q)^2$:
$$(x,y)\sim (x',y') \ \text{iff} \ \exists g\in SL(2,\ov\BF_q) \ | \ x'=gxg^{-1}, \
y'=gyg^{-1}.$$
  Then every absolutely irreducible component of the set of
  conjugacy classes of $\tilde\vp$-periodic points   is positive dimensional.
 \end{cor}
\begin{proof}
  Indeed  $(SL(2,\BF_q)^2\setminus V( \BF_{q})) /_{\sim}$
  can be identified with
$\BF_{q}^3.$
The corollary is valid, because
 the periodic set of the trace map is positive dimensional.\end{proof}


We can now obtain a sufficient condition for the existence of
periodic points. Consider the  maps $\vp\colon \tG\times \tG \to
\tG$ and $\psi\colon \BA_{s,u,t}^3\to\BA_{s,u,t}^3$ as in diagram
(\ref{d4}), and denote by
$\Phi\subset \BA_{s,u,t}^3$ the variety of fixed points of $\psi$.
As in Section \ref{np}, for a fixed $y$ denote by $\vp_y\colon
\tG\to\tG$ the map $x\mapsto \vp (x,y)$.

Note that $\Phi$ contains a line $$L_1=\{s=2, u=t\}.$$ Since $\Phi$
is a complete intersection, all its irreducible components have
dimension at least one.

\begin{Theorem}\label{condition}
Write $\Phi=\cupl_{i=1}^k W_i\cup L_1,$
where $W_i$ are irreducible $\BF_q$-components of $\Phi.$
Suppose $q$ is big enough. If at least one of $W_i$'s is absolutely irreducible,
then there exists a pair
$(x,y)\in G\times G$ such that $x\ne 1, y\ne 1$ and $x$
is a periodic point of $\vp_y.$
\end{Theorem}

\begin{proof}
Let $W_i$ be an absolutely irreducible component of $W$, $W\ne L_1$.
By the Lang--Weil theorem \cite{LW}, there is a point
$Q=(s_0,u_0,t_0)\ne(\pm 2,t,\pm t)\in W_i(\BF_q).$ According to
\thmref{rational}, we have $H_Q(\BF_q)\ne\emptyset,$ where
$H_Q=\pi^{-1}(Q).$ It follows that there exists a pair $(x,y)\in
\tG\times \tG$ such that $s_0=\tr(x)$, $u_0=\tr(xy)$, $t_0=\tr(y).$
By \lemref{fix},  $(\vp_y(x),y)\in H_Q(\BF_q)$ as well. Since the
set $H_Q(\BF_q)$ is finite, there are numbers $n<m\in\BN$ such that
$\vp_y^{(m)}(x)=\vp_y^{(n)}(x).$ Thus, $\tilde x=\vp_y^{(n)}(x)$ is
a periodic point of $\vp_y.$ Moreover, the image of $\tilde x$ in
$G=PSL(2,q)$ is non-identity since $Q=(s_0,u_0,t_0)\ne (\pm 2,t,\pm
t)$.
\end{proof}

\begin{Remark}
If there is a component $W_i\subset \Phi$ defined over $\BZ$ and
irreducible over $\ov\BQ,$ then, by
\cite [Theorem IV, 9, 7.7(i)]{Gr}, the assumptions of the theorem
are satisfied for any prime $p$ big enough.
\end{Remark}

\begin{remark} \label{rem:int}
Suppose $q=p>3$ is a prime number. Note that all the maps in diagram
(\ref{d4}) are defined over $\BZ$, and it can thus be viewed as the
special fibre at $p$ of the following diagram of morphisms of
$\BZ$-schemes (denoted by the same letters):

\begin{equation}
\begin{CD}
\CG \times \CG   @>{\tilde\vp}>>   \CG \times \CG   \\
@V\pi VV @V\pi VV \\
\BA^3_{\BZ} @>{\psi}>> \BA^3_{\BZ}
\end{CD}
\label{d4Z}
\end{equation}
where $\CG=SL(2,\BZ )$.
\end{remark}

\subsection{Three-variable maps} \label{3vm}

Let here $\tG$ denote $SL(2,K)$ where $K$ is an arbitrary field. 
Consider a  morphism $\vp\colon \tG\times \tG\times \tG\to \tG$ such that
$$\vp(\pm x,\pm u,\pm y)=\pm\vp(x,u,y).$$ The modified map $\tilde
{\vp}\colon \tG \times \tG  \times \tG \to \tG\times \tG  \times \tG
$ is defined by $\tilde {\vp}(x,u,y)=(x,\vp(x,u,y),y)$.

As above, we consider a representation $\rho$ of the free group
$F_3$ in $SL(2,\BZ)$ and assume that $\varphi$ is defined by a word 
$w=w(x,u,y)$. The trace of $\rho (w)$
can be expressed as a polynomial in 7 variables $a_1=\tr(x)$,
$a_2=\tr(y)$, $a_3=\tr(u),$ $a_{12}=\tr(xy)$, $a_{13}=\tr(xu),$
$a_{23}=\tr(yu)$, $a_{123}=\tr(xyu)$. These variables are dependent
(see, e.g., \cite{Ma1} or  formulas (2.3)--(2.5) in \cite{Ho}):
\begin{equation}\label{hha1}
\begin{aligned}
a_{123}^2 & - a_{123}( a_{12} a_3+ a_{13} a_2+a_{23}
a_1-a_1a_2a_3)\\
& + (a_1^2+a_2^2+a_3^2+a_{12}^2+a_{13}^2+a_{23}^2 -a_1 a_2 a_{12}-
a_1 a_3 a_{13}-a_2 a_3 a_{23}+a_{12} a_{13}a_{23}-4) =0.
\end{aligned}
\end{equation}

Let $\ov{a}= (a_1, a_2,a_3, a_{12},a_{13},a_{23},a_{123})\in \BA^7,$
let $Z\subset \BA^7$ be an absolutely irreducible set  defined by \eqref{hha1}. Let
$\pi(x,u,y)=\ov{a}\in Z$  be the trace projection. Then the
following diagram is commutative:
\begin{equation}
\begin{CD}
\tG \times \tG  \times \tG   @>{\tilde\vp}>>  \tG \times \tG\times \tG  \\
@V\pi VV @V\pi VV \\
Z(K) @>{\psi} >> Z(K)
\end{CD}
\label{d6}
\end{equation}
where $\psi(\ov a)=(a_1, a_2 ,l_1(\ov a), a_{12}, l_2(\ov a),l_3(\ov
a),l_4(\ov a))$,
$$l_1=\tr(\vp(x,u,y)), \ l_2=\tr(\vp(x,u,y)x),$$
$$l_3=\tr(\vp(x,u,y)y), \  l_4=\tr(\vp(x,u,y)xy).$$
The variety $F(\vp)\subset Z$ of fixed points of $\psi$ is defined
by the equations
$$l_1(\ov a)=a_3, \ l_2(\ov a)=a_{13}, \ l_3(\ov a)=a_{23}, \ l_4(\ov a)=a_{123},$$
and, since it is nonempty, its dimension is at least $3$.

Let us now consider diagram (\ref{d6}) more carefully.

\begin{Lemma}\label{Zirred}
Let $F$ be any algebraically closed field. Then the set $Z$ is an irreducible hypersurface over $F$.  
\end{Lemma}

\begin{proof} Assume the contrary. Let $p$ denote the  natural
projection of $\BA^7$   to $ \BA^6, $ forgetting the coordinate $a_{123}.$ Let $L\subset \BA^6$ be 
an irreducible curve not contained in the branch locus of the restriction of $p$ to $Z$. 
Then the set $p^{-1}(L)\bigcap Z$ is reducible. 

\medskip
{\bf Case 1.}  $char (F)\ne 2.$ 
 
Let $c\ne\pm 2$. Consider  the curves
$L=\{a_1=a_2=a_{13}=a_{23}=0, \, a_{12}=c \}\subset \BA^6$  and $M=p^{-1}(L)=\{a_1=a_2=a_{13}=a_{23}=0, \, a_{12}=c\}\subset \BA^7.$
Then from \eqref{hha1} it follows that $  M'=Z\bigcap M$  is defined by the following equations:
$$
(a_{123}-a_3c/2)^2-(c^2-4)(a_{3}^2-4)/4=0, \, 
 a_1=a_2=a_{13}=a_{23}=0,\, a_{12}=c.
$$
Therefore $M'$ is a branched double cover of $L,$ hence it is irreducible. Contradiction.
 
 \medskip
  {\bf Case 2.}  $char (F) =2.$ 
  We now consider the curve $L=\{a_1=a_2=a_{13}=a_{23}=0, \, a_{12}=a_3+1 \}\subset \BA^6.$ 
 In the notation of Case 1, $M'$ is defined by the equations 
 $$a_{123}^2-a_3(a_3+1)a_{123}+1=0,\, 
 a_1=a_2=a_{13}=a_{23}=0,\, a_{12}=a_3+1.$$
  Thus it is irreducible. Contradiction. 

\medskip

Hence $Z$ is irreducible.   
\end{proof}

\begin{Theorem}\label{surj}
Let $Z\subset \BA^7_{a_1, a_2,a_3, a_{12},a_{13},a_{23},a_{123}}$ be
defined by equation \eqref{hha1}. Then for all  $q$ 
the map $\pi\colon SL(2,q)\times SL(2,q)\times SL(2,q)\to  Z(\BF_q)$ 
is surjective. 
\end{Theorem}

\begin{proof} 
The result will follow from identities
between certain polynomials in the polynomial ring
$$R:=\BZ[x_1, x_2,x_3, x_{12},x_{13},x_{23},x_{123},\a_1,\g_1,\a_2,\g_2].$$
Denote 
\begin{equation}\label{hhaf}
\begin{aligned}
L:= & x_{123}^2  - x_{123}( x_{12} x_3+ x_{13} x_2+x_{23}+
x_1-x_1x_2x_3)\\
&  +x_1^2+x_2^2+x_3^2+x_{12}^2+x_{13}^2+x_{23}^2 -x_1 x_2 x_{12}-
x_1 x_3 x_{13}-x_2 x_3 x_{23}+x_{12} x_{13}x_{23}-4,
\end{aligned}
\end{equation}
$$L_{12}:=x_1^2+x_2^2+x_{12}^2-x_1x_2x_{12}-4,\quad
L_{13}:=x_1^2+x_3^2+x_{13}^2-x_1x_3x_{13}-4,$$
$$L_{23}:=x_2^2+x_3^2+x_{23}^2-x_2x_3x_{23}-4$$
(all viewed as elements of $R$).

We start with the following lemma (skipping an elementary proof).  
\begin{lemma}\label{quale}
Let $K$ be a finite field, and let $r,\, s,\, t,\, a\in K$ be such that the equation in $x$, $y$ 
$$x^2+y^2+rxy+sx+ty=a$$
is not solvable in $K$. Then the characteristic of $K$ is 
$2$ and $r=0$, $s=t$ hold. \qed
\end{lemma}

We now define two more polynomials in the ring $R$ (the reason will 
become clear later on): 
$$D_1:=-\a_1^2 + \a_1\g_1x_3 + \a_1x_1 - \g_1^2 - \g_1x_1x_3 + \g_1x_{13} - 1,$$
$$D_2:=-\a_2^2 + \a_2\g_2x_3 + \a_2x_2 - \g_2^2 - \g_2x_2x_3 + \g_2x_{23} - 1.$$
Our argument will also need the following two by two matrix over $R$: 
\begin{equation}\label{pm5}
A=\left( \begin {array}{cc}
      2\a_2 - \g_2x_3 - x_2          &  -\a_2x_3 + 2\g_2 + x_2x_3 - x_{23}   \\
 \a_2x_3 - 2\g_2 - x_2x_3 + x_{23}& -\a_2x_3^2+2\a_2+\g_2x_3+x_2x_3^2-x_2- x_3x_{23}
        \end {array} \right).\qquad
\end{equation}
Define further $\tilde A$ to be the adjoint matrix of $A$, that is 
$\tilde A$ is $A$ with the diagonal entries permuted and the 
off-diagonal entries multiplied by $-1$. The product
$\tilde A A$ is the scalar matrix corresponding to the determinant of $A$.
We further consider the vector
$$
b:= \left( \begin {array}{c}
\a_2x_1- \g_2x_1x_3 + \g_2x_{13} - x_1x_2 + x_{12}      \\
                \a_2x_{13} - \g_2x_1 - x_2x_{13} + x_{123}
        \end {array} \right)\in R^2
$$
and define $r, s\in R$ by 
$$\left( \begin {array}{c}
                 r\\
                s
        \end {array} \right):=\tilde A b.
$$
Multiply now $D_1$ by $L_{23}^2$ and replace $y_1:=L_{23}^2\a_1$, 
$y_2:=L_{23}^2\g_1$, obtaining the polynomial
$$F(y_1,y_2):=
-y_1^2 + y_1y_2x_3 + y_1L_{23}x_1 - y_2^2 - y_2L_{23}x_1x_3 + 
y_2L_{23}x_{13} -L_{23}^2$$
in the variables $y_1,\, y_2$.

We need one more lemma. 
\begin{lemma}\label{qual}
Let ${\mathfrak D}_2$ be the ideal of $R$ generated by $D_2$ and ${\mathfrak D}$ the ideal
generated by $D_2$ and $L$. 
Then the following hold:
\begin{itemize}
\item[{(i)}] $\det(A)-L_{23}$ is in ${\mathfrak D}_2$;
\item[{(ii)}] $F(r,s)$ is in  ${\mathfrak D}$. 
\end{itemize} 
\end{lemma}
The proof of this lemma amounts to certain simple computations   
which are best done using a computer algebra system. The first item 
follows for example from the identity:
$$\det(A)-L_{23}=(x_3^2-4)D_2.$$
For the second item, the formula is more complicated. We skip the details.
\qed

\bigskip

We can now go over to the proof of the theorem. 

Let $K$ be any field. Let  
$x=(x_1, x_2,x_3, x_{12},x_{13},x_{23},x_{123})\in Z(K)$.
As we are working with traces and are thus allowed to make 
simultaneous conjugation, we start our search of solutions to $\pi (B_1, B_2, B_3)=x$  
by considering the following triples of two by two matrices over 
the polynomial ring $K[\a_1,\g_1,\a_2,\g_2]$:
\begin{equation}\label{pm1}
B_1=\left( \begin {array}{cc}
\a_1 & -\a_1x_3+ \g_1 + x_1x_3 - x_{13}\\
\a_1 & x_1-\a_1
\end {array} \right),\qquad
B_2=\left( \begin {array}{cc}
\a_2 & -\a_2x_3+ \g_2 + x_2x_3 - x_{23}\\
\g_2 & x_2-\a_2
\end {array} \right),\qquad
\end{equation}
\begin{equation}\label{pm2}
B_3=\left( \begin {array}{cc}
0 & 1\\
-1 & x_3
\end {array} \right).\qquad
\end{equation}

The condition that $B_1,\, B_2,\, B_3$ have determinant $1$ and satisfy 
$\pi(B_1, B_2, B_3)=x$ is equivalent to the four equations: 
\begin{equation}
D_1=-\a_1^2 + \a_1\g_1x_3 + \a_1x_1 - \g_1^2 - \g_1x_1x_3 + \g_1x_{13} - 1=0,
\end{equation}
\begin{equation}\label{dete2}
D_2=-\a_2^2 + \a_2\g_2x_3 + \a_2x_2 - \g_2^2 - \g_2x_2x_3 + \g_2x_{23} - 1=0,
\end{equation}
\begin{equation}
\begin{aligned}
\a_1(2\a_2 - \g_2x_3 - x_2) & +
\g_1(- \a_2x_3 + 2\g_2 + x_2x_3 - x_{23}) \\
& - \a_2x_1+ \g_2x_1x_3 - \g_2x_{13} + x_1x_2 - x_{12}=0,
\end{aligned}
\end{equation}
\begin{equation}
\begin{aligned}
\a_1(\a_2x_3 - 2\g_2 - x_2x_3 + x_{23}) & +
\g_1(-\a_2x_3^2 + 2\a_2 + \g_2x_3 +x_2x_3^2 - x_2 - x_3x_{23}) \\
& - \a_2x_{13} + \g_2x_1 + x_2x_{13} - x_{123}=0. 
\end{aligned}   
\end{equation}
Notice that the first equation is quadratic in $\a_1,\g_1$ only and the second
is quadratic in $\a_2,\g_2$ only. The third and fourth equations are written as a 
linear system in $\a_1,\g_1$. Defining
the vectors
$$y:= \left( \begin {array}{c}
                 \a_1\\
                \g_1
        \end {array} \right),\qquad
b:= \left( \begin {array}{c}
\a_2x_1- \g_2x_1x_3 + \g_2x_{13} - x_1x_2 + x_{12}      \\
                \a_2x_{13} - \g_2x_1 - x_2x_{13} + x_{123}
        \end {array} \right),\qquad
$$
the third and fourth of the above equations can be schematically written as
$$Ay=b$$
with the matrix $A$ defined in (\ref{pm5}) evaluated at our point $x$. 

We now assume that $K$ is a finite field.
We shall now write $L_{23}(x)$ for the polynomial $L_{23}$ defined above 
evaluated at our point $x\in Z(K)$, that is 
$L_{23}(x)=x_2^2+x_3^2+x_{23}^2-x_2x_3x_{23}-4$. We use a 
similar notation for all the other polynomials.

{\bf Case 1: At least one of the values $L_{12}(x)$, $L_{13}(x)$, $L_{23}(x)$ is nonzero.} 
Assume, say, $L_{23}(x)\ne 0$ (the other cases are similar). 

First we show that (\ref{dete2}), viewed as an equation in the indeterminates
$\a_2$, $\g_2$,  has a solution. Assume the contrary. Then by Lemma  
\ref{quale} we conclude that the characteristic of $K$ is two, 
$x_3=0$ and $x_2=x_{23}$. This contradicts the assumption 
$L_{23}\ne 0$.

We shall now fix a solution $(\a_2,\g_2)\in K^2$ of  equation (\ref{dete2}) 
and put these into the above matrix $A$ 
getting a two by two 
matrix over $K$. Similarly we get a vector $b$ in $K^2$. 
By Lemma \ref{qual} we find
$$\det(A)=L_{23}(x)\ne 0$$ 
which is guaranteed by our assumption. We now define $(\a_1,\g_1)\in K^2$
by
$$\left( \begin {array}{c}
                 \a_1\\
                \g_1
        \end {array} \right):=A^{-1} b.$$
By Lemma \ref{qual}(ii), we have found three matrices
 $B_1,\, B_2,\, B_3\in SL(2,K)$ satisfying
$\pi(B_1, B_2, B_3)=x$.

If now $L_{23}(x)=0$, we have either $L_{12}(x)\ne 0$ or $L_{13}(x)\ne 0$. 
These cases are treated in a similar way. \qed

\begin{remark} \label{rem:quad-closed}
The above proof remains true if $K$ is any quadratically closed field (cf. also \cite{Pe1}). 
\end{remark}

{\bf Case 2: $L_{12}(x)=L_{13}(x)=L_{23}(x)=0$}. 

Loosely speaking, our strategy in this case is to use automorphisms of the free group $F_3$ 
to get from $x$ another point of $Z(K)$ such that not all three values of $L_{ij}$ vanish 
at that point, and then use the result of Case 1. Let us make this more precise. 

We start with an obvious lemma. 

\begin{lemma} \label{lem:sym}
Let $n\ge 2$, let $F_n$ denote the free group on $n$ generators $X_1$,\dots, $X_n$, and let $G^n$ be the product of $n$ copies 
of a group $G$. The map 
$$
\Aut(F_n)\to \Sym(G^n), \quad \vp\mapsto\hat\vp , 
$$
defined by 
$$
\hat\vp (T)=(\vp (X_1)_T,\dots ,\vp (X_n)_T),
$$
is a group homomorphism. 

Here $T$ is an $n$-tuple of elements of $G$ and $\vp(X_i)_T$ is the element of $G$
obtained by substitution of the elements of $T$ instead of the $X_i$'s appearing 
in the expression of $\vp (X_i)$ in the basis $\{X_1,\dots ,X_n\}$.  \qed
\end{lemma}

The following constructions are described in \cite{Ho} (see also \cite{Ma1}, \cite{Pe2}), 
sometimes with details omitted. For the reader's convenience and sake of completeness 
we now focus on the case $n=3$ giving some more details. Fix a basis $\{X,Y,Z\}$ of $F_3$. 

\begin{Definition} \label{def:action}
For every $\vp\in \Aut(F_3)$ define a map $F_{\vp}\colon \BA^7\to\BA^7$ by the formula 
$$
F_{\vp}(u):=(P_{\vp (X)}(u), P_{\vp (Y)}(u),\dots ,P_{\vp (XYZ)}(u)), 
$$
where $P_w$ is the integer polynomial in $7$ variables corresponding to the word $w$ 
(cf. Theorem $\ref{Hor}$).
\end{Definition}


\begin{lemma} \label{lem:pi-phi}
For every $\vp\in\Aut(F_3)$ and every $T\in SL(2,K)^3$ we have 
$$
\pi (\hat\vp (T))=F_{\vp}(\pi (T)). 
$$
\end{lemma}
\begin{proof} Obvious. 
\end{proof}
\begin{lemma} \label{lem:Z}
For every $\vp\in\Aut(F_3)$ and every field $K$ we have $F_{\vp }(Z(K))\subseteq Z(K)$. 
\end{lemma}
\begin{proof} 
We first prove that 
$F_{\vp }(Z(\ov {K}))\subseteq Z(\ov{K}),$ where $\ov {K}$ is an algebraic closure of $K.$
From this the needed inclusion will follow 
as soon as $F_{\vp }$ is defined over $K.$
In Case 1 we have proven that the map $\pi$ is surjective  onto an open subset 
$$U(\ov {K})=\{L_{12}\ne 0,L_{13}\ne 0,
L_{23}\ne 0\}\subseteq  Z(\ov K),$$
since the proof was valid for any algebraically closed field (see Remark \ref{rem:quad-closed}).

Let $u\in U(\ov {K}), u=\pi(T),  T\in SL(2,\ov K)^3.$ 
Then $F_{\vp}(u)=F_{\vp}(\pi (T))=\pi (\hat\vp (T))\in Z(\ov {K}).$
Hence, $F_{\vp }(U(\ov {K}))\subseteq Z(\ov{K}).$
Since $U$ is open in $Z$ and $Z$ is irreducible, the same inclusion is valid for $Z.$
Since $F_{\vp }$ is defined over $\BZ,$ the inclusion for $K$-points follows as well. 
\end{proof} 

\begin{lemma} \label{lem:hom}
\begin{itemize}
\item[{(i)}] $F_{\text{\rm{id}}}=\text{\rm{id}}$;
\item[{(ii)}] For every $\vp, \psi\in\Aut (F_3)$ and every $u\in Z(K)$ we have
$$
F_{\vp\circ\psi}(u)=F_{\vp}\circ F_{\psi}(u).
$$
\end{itemize}
\end{lemma}
\begin{proof} 

The first item is obvious, so let us prove the second one. 
Once again, similarly to \lemref{lem:Z}, it is sufficient to prove it 
over an open subset $U$ considered in \lemref{lem:Z}, 
and over the algebraically closed field $\ov {K}).$

Let  us take $u\in U(\ov {K}), u=\pi(T),  T\in SL(2,\ov K)^3.$
Using Lemmas \ref{lem:sym} and \ref{lem:pi-phi}, we get
$$
F_{\vp\circ\psi}(u)=\pi (\widehat{\vp\circ\psi}(T))=\pi (\hat\vp\circ\hat\psi(T)), 
$$
$$
F_{\vp}\circ F_{\psi}(u)=F_{\vp}(\pi(\hat\psi(T))=\pi (\hat\vp(\hat\psi(T))), 
$$
so the needed equality is proved. \end{proof}

\begin{cor} \label{cor:repr}
The correspondence $\vp\mapsto F_{\vp }$ defines a group homomorphism 
$\Aut (F_3)\to \Aut (Z)$ where $\Aut (Z)$ is the group of $\mathbb Z$-defined 
polynomial automorphisms of the variety $Z$. \qed
\end{cor}

We can now go over to the proof of the theorem in Case 2. 

Let, as above, $x\in Z(K)$ be such that $L_{12}(x)=L_{13}(x)=L_{23}(x)=0$. 

\medskip

{\bf Case 2a}. 
Let first assume that there exists $\vp\in\Aut (F_3)$ such that $u:=F_{\vp }(x)$ 
is such that not all three values  $L_{12}(u)$, $L_{13}(u)$, $L_{23}(u)$ are zero. 
By Case 1, there exists $T\in SL(2,K)^3$ such that $\pi (T)=u$. Define $T':=\hat\vp^{-1}(T)$. 
By Lemma \ref{lem:pi-phi} and Corollary \ref{cor:repr}, we have $\pi (T')=F_{\vp^{-1}}(\pi (T))=
F_{\vp^{-1}}(u)=F_{\vp}^{-1}(u)=x$, and we are done. 

\medskip

{\bf Case 2b.} Assume that there is no such $\vp$ as in Case 2a. 


Denote by $L_{ij}^{\vp}$ (where $i,j$ stand for distinct numbers from 
the set $\{1,2,3\}$) the polynomials in 7 variables 
obtained after applying $F_{\vp}$ to $L_{ij}$. The needed contradiction immediately follows 
from the following proposition. 

\begin{Proposition} \label{prop:case2b}
Denote the automorphisms of $F_3$ sending the basis $\{X,Y,Z\}$ to the 
bases $\{XY,Y,Z\}$, $\{X,YZ,Z\}$, $\{X,Y,XZ\}$, $\{XY^{-1},Y,Z\}$, $\{X,Y,YZ\}$, $\{XY^2,Y,Z\}$, 
$\{X,ZYZ^{-1},Z\}$, $\{X,Y,XZX^{-1}\}$, by $\vp_1,\dots ,\vp_8$, respectively.  Denote by $\mathfrak a$ the ideal in 
$\mathbb Z[x_1,\dots ,x_{123}]$ generated by the functions $L_{ij}^{\vp_m}$ where, as above, 
$i,j$ stand for distinct numbers from 
the set $\{1,2,3\}$, and $k=1,\dots ,8$, and let 
$$Z_{\mathfrak a}(K)=\{x\in \BA^7(K): f(x)=0 {\text{\rm{ for all }}} f\in\mathfrak a\}.$$
Then for any field $K$ of characteristic different from $2$ we have 
$$
\begin{aligned}
Z_{\mathfrak a}(K) = & \{(2,2,2,2,2,2,2), (0,-2,-2,0,0,2,0), (0,-2,2,0,0,-2,0), (0,2,-2,0,0,-2,0), \\ 
&(0,2,2,0,0,2,0), (0,0,0,-2,-2,-2,0), (0,0,0,-2,2,2,0), (0,0,0,2,-2,2,0), \\ 
& (0,0,0,2,2,-2,0)\}, 
\end{aligned}
$$ and for any field of characteristic $2$ 
we have $Z_{\mathfrak a}(K)=\{(0,0,0,0,0,0,0), (1,0,0,1,1,0,1)\}$.   
\end{Proposition}

\begin{proof}
MAGMA computation. 
\end{proof}

For each of the points $x$ appearing in Proposition \ref{prop:case2b} one can easily exhibit an explicit triple 
of matrices $T$ such that $\pi (T)=x$. Say, $\pi (\Idd ,\Idd , \Idd )=(2,2,2,2,2,2,2)$, $$ \pi \left(\left(
\begin{matrix}
0 & -1 \\
1 & 0
\end{matrix}
\right), \, 
\left(
\begin{matrix}
-1 & 0 \\
0 & -1
\end{matrix}
\right), \, 
\left(
\begin{matrix}
-1 & 0 \\
0 & -1
\end{matrix}
\right)
\right)
=(0,-2,-2,0,0,2,0),$$ and so on.

Theorem \ref{surj} is proved. 
\end{proof}

Coming back to the map $\tilde{\vp},$ let us  consider an additional condition:
\begin{equation}\label{cond}
u=w(x,y),
\end{equation}
where $x\in \tG, y\in \tG$ and $w\in F_2$. Let $g_3(a_1,a_2,a_{12})=\tr(w(x,y)),$ 
$g_{13}(a_1,a_2,a_{12})=\tr(w(x,y)x),$ 
$g_{23}(a_1,a_2,a_{12})=\tr(w(x,y)y),$
$g_{123}(a_1,a_2,a_{12})=\tr(w(x,y)xy).$ Then \eqref{cond} defines a
three-dimensional variety $W(w)\subset Z$:
\begin{equation} \label{Ww}
W(w)=Z\bigcap\left\{\begin{aligned}
a_3&=g_3(a_1,a_2,a_{12}),\\
             a_{13}&=g_{13}(a_1,a_2,a_{12}),\\  a_{23}&=g_{23}(a_1,a_2,a_{12}),\\
a_{123}&=g_{123}(a_1,a_2,a_{12})\end{aligned}\right \}.
\end{equation}

We can now formulate a result which treats the $SL(2,q)$-case for
three-variable maps and thus makes a crucial step towards getting a
sufficient condition for a given sequence of type (\ref{eq:ours-it})
to characterize finite solvable groups.

\begin{Theorem}\label{s3vm}
Let $v(x,u,y)$ and $w(x,y)$ be words in the free groups with three
and two generators, respectively. Define a sequence $u_n(x,y)$ by
the following recurrence relations:
$$u_0(x,y)=w(x,y), \ u_{n+1}(x,y)=v(x,u_n(x,y),y).$$

Let $\vp\colon \tG\times \tG\times \tG\to \tG$ be the map defined by
$(x,u,y)\mapsto v(x,u,y)$, let $F(\vp)$ be the variety of fixed
points of the trace map $\psi$ induced by $\vp$ (see diagram
$(\ref{d6})$), and let $W(w)$  be defined by $\eqref{Ww}$.  
With the notation of Theorem $\ref{surj}$, let 
$V=\{a_2=2, a_1=a_{12}, a_3=a_{23}, a_{13}=a_{123}\}$. 

Assume that $F(\vp)\bigcap W(w)$ contains a positive dimensional,
absolutely irreducible $\BQ$-sub\-variety $\Phi$ such that
$\Phi':=\Phi\setminus(\Phi\bigcap V)$ 
is an open subset of $\Phi.$

Then there is $q_0$ such that for every  $q>q_0$ there exists a pair $(x,y)\in\tG\times
\tG$ with $u_n(x,y)\ne 1$ for all $n\in\BN.$
 \end{Theorem}

\begin{proof}  Let $q_0$ be such that
 $\Phi'(\BF_q)\neq\emptyset.$
Let $\ov a\in \Phi'(\BF_q).$ 
By \thmref{surj}, there is a triple
$(x,u,y)\in \tG\times \tG\times \tG$ such that $\pi(x,u,y)=\ov a.$
Moreover, since $\ov a\in W(w),$ we may take $u=w(x,y).$ Since $\ov
a\in \Phi$, we have $\psi(\ov a)=\ov a,$ hence $
\pi(x,u_1(x,y),y)=\ov a.$ Similarly, $ \pi(x,u_n(x,y),y)=\ov a$ for
all $n\in \BN.$

Since 
$a_2=\tr u_n(x,y)\ne 2$, we have $u_n(x,y)\ne 1.$
\end{proof}

\begin{Remark} \label{SL-PSL}
Although this section was completely devoted to considering the group $SL(2)$ 
(until now $PSL(2)$ only appeared in its title), the obtained results (in particular, 
Theorems \ref{condition} and \ref{s3vm}) are also applicable to the $PSL(2)$-case. 
(In the two-variable case, this is explicitly explained at the end of the proof of 
Theorem \ref{condition}, the case of Theorem \ref{s3vm} is similar). 
\end{Remark}

\section{Case $G=Sz(q)$} \label{sec:Suzuki}



In this section we consider a map $\vp\colon G\times G\to G$ where
$G$ is a Suzuki group, $Sz(q)$, $q=2^{2m+1}$, $m\ge 1$. As above,
for a fixed $y\in G$ we denote by $\varphi _y\colon G\to G$ the map
$(x,y)\mapsto \vp(x,y).$  There is no trace map in this case.
Nevertheless there is a factorization (see diagram (\ref{dd})) which
simplifies the picture. This leads to a sufficient condition
(\thmref{ScSz}) for the existence of periodic points.  Although the
condition is not that simple, we have an example in
\subsecref{subsec:BWW} when it works.

Recall that according to the Bruhat decomposition, $G=U_1\cup U_2, $
where the first Bruhat cell $U_1=B$  consists of all
lower-triangular matrices of the form $x=T(a,b)D(k)$ with
$$
T(a,b)=\begin{pmatrix}   1 & 0 & 0 & 0\\
a & 1& 0 & 0\\
a^{1+s}+b & a^s & 1& 0\\
a^{2+s}+ab+b^s& b & a & 1\end{pmatrix},
$$
\
$$
D(k) = \begin{pmatrix} k^{s/2+1} & 0 & 0 & 0\\
0 & k^{s/2} & 0 & 0\\
0 & 0 & k^{-s/2} & 0\\
0 & 0 & 0 & k^{-s/2-1}\end{pmatrix},
$$
and the second Bruhat cell $U_2$ consists of the matrices

\begin{equation}x=T(a,b)D(k)wT(c,d),\end{equation}
where
$$
w = \begin{pmatrix}
0 & 0 & 0 & 1\\
0 & 0 & 1 & 0\\
0 & 1 & 0 & 0\\
1 & 0 & 0 & 0\end{pmatrix}.
$$
Here $a,b\in \BF_q$, $k\in\BF^*_q$, $s=2^{m+1}$.

Recall the following properties of these matrices:
\begin{itemize}
\item[(i)]  $T(0,1)T(a,b)=T(a,b)T(0,1)$;
\item[(ii)] $D(k)w=wD(k^{-1})$;
\item[(iii)] $T(a,b)T(c,d)=T(a+b,ac^s+b+d)$;
\item[(iv)] $wT(0,t)w=T(t^{1-s},t^{-1})D(t^{2s/(s+2)})wT(t^{1-s},0)$;
\item[(v)] $T(0,1)^{-1}=T(0,1)$;
\item[(vi)] $D(k)^{-1}T(a,b)D(k)=T(ak,bk^{1+s})$.
\end{itemize}

For $x=T(a,b)D(k)wT(c,d)\in U_2$ define
$$x'=\varkappa(x)=T(c,d)xT(c,d)^{-1}=T(c,d)T(a,b)D(k)w=T(a+c,ca^s+b+d)D(k)w.$$

Note that for any $z=T(\alpha,\beta)$ we have
\begin{align*}\varkappa(zxz^{-1})&=\varkappa(T(\alpha,\beta)T(a,b)D(k)wT(c,d)T(\alpha,\beta)^{-1})
\\ &=T(c,d)T(\alpha,
\beta)^{-1}T(\alpha,\beta)T(a,b)D(k)w=T(c,d)T(a,b)D(k
)w=\varkappa(x).\end{align*}

\begin{lemma}\label{S}  If for any  $y,x,h\in G$ we have
\begin{equation}\label{ce1}
\varphi _y (hxh^{-1})=h\varphi _{hyh^{-1}}(x)h^{-1},
\end{equation}
then for $y=T(0,t)$ we have
$$\varphi _y(\varkappa(x))=\varkappa(\varphi
_y(x)).$$
\end{lemma}

\begin{proof} For $z=T(c,d)$ we have
$$\varphi_y(zxz^{-1})=z
\varphi_{z^{-1}yz}(x)z^{-1}.$$ Since the matrices $T(0,t)$ commute
with any $z,$ it follows that
$$\varphi_y(zxz^{-1})=z\varphi_y(x)z^{-1},$$ i.e.
$$\varphi_y(\varkappa(x))=\varkappa(z\varphi_y(x)z^{-1})=\varkappa(\varphi_y(x)).$$
\end{proof}

From now on until the end of this section we only consider elements
$x$ from the second Bruhat cell.

\begin{cor}\label{com} For $x\in U_2$ denote $\pi_1(x)=a+c,$\ $\pi_2(x)=ca^s+b+d,$\ $k(x)=k.$
Then for $y=T(0,t)$ there exist functions $f,g$ and $h$ such that if
$\varphi_y(x)\ne 1$ then
$$\pi_1(\varphi_y(x))=f(\pi_1(x),\pi_2(x),k(x)),$$
$$\pi_2(\varphi_y(x))=g(\pi_1(x),\pi_2(x),k(x)),$$
$$k(\varphi_y(x))=h(\pi_1(x),\pi_2(x),k(x)).$$
\end{cor}

\begin{proof} Indeed, by construction $\varkappa(x)=T(\pi_1(x),\pi_2(x))
D(k(x))w.$ Thus by \lemref{S}, we have
$$T(\pi_1(\varphi_y(x)),\pi_2(\varphi_y(x))D(k(\varphi_y(x))w=\varkappa(\varphi_y(x))=\varphi_y
(\varkappa(x))=\varphi_y(T(\pi_1(x),\pi_2(x))D(k(x))w).$$ It follows
that $\pi_1(\varphi_y(x)),\pi_2(\varphi_y(x))$ and $k(\varphi_y(x))$
are determined uniquely by the values of $\pi_1(x),$\ $\pi_2(x)$ and
$k(x).$
\end{proof}


\corref {com} may be expressed by the following commutative diagram
of $\BF_q$-morphisms:
\begin{equation}
\begin{CD}
\BA^2_{a,b}\times\BA^*_{k}\times\BA^2_{c,d}\supseteq U @>\varphi_y>>\BA^2_{a,b}\times\BA^*_{k}\times\BA^2_{c,d}\\
@V\pi VV @V\pi VV \\
\BA^2_{a,b}\times\BA^*_{k}@> \psi >>\BA^2_{a,b}\times\BA_{k}^*\\
\end{CD}
\label{dd}
\end{equation}
where $U$ denotes  the set of $x\in U_2$ such that
$\varphi_y(x)\ne 1.$

This corollary provides the following sufficient condition for the
existence of periodic points which can be viewed as an analogue of
Theorem \ref{condition}:

\begin{Theorem}\label {ScSz}
Let $G=Sz(q)$, let $y=T(0,1)\in G$, and suppose that the map $\vp_y$
satisfies the following conditions:
\begin{itemize}
\item\label{z1} equality \eqref{ce1} holds for any $x,y,h\in G$;
\item\label{z2} the morphism $\psi\colon \BA^2_{a,b}\times\BA^*_{k}\to
\BA^2_{a,b}\times\BA^*_{k}$ induced by $\vp_y$ (see diagram $(\ref{dd})$)
has an invariant set $V$ (i.e. $\psi(V)\subset V$).
\end{itemize}
Then the map $\vp_y\colon G\to G$ has a non-identity periodic point.
\end{Theorem}

\begin{proof} Indeed, the cell $U_2$ does not contain the identity
matrix.
\end{proof}

\begin{Remark}\label{rmk}
In view of \eqref{ce1}, the statement of Theorem \ref{ScSz} holds
for any $y=T(0,t).$
\end{Remark}

\section{Examples}\label{Examples}

In this section we want to demonstrate how the trace map works. In
\subsecref{subsec:BWW} we consider the two-variable case and give a
new proof of the main theorem of \cite{BWW} characterizing finite
solvable groups. In \subsecref{un} we compute the trace map for the
three-variable sequence from \cite{BGGKPP1}, \cite{BGGKPP2} (that
also characterizes finite solvable groups). In \subsecref{y2} we
apply our method for finding a modified sequence having the same
property. 
Subsection \ref{Com} contains an illustration of the
method for a simple case where the word under consideration is
commutator.

\subsection{The sequence of Bray--Wilson--Wilson}\label{subsec:BWW}

The sequence $s_n(x,y)$ of \cite{BWW} is defined as follows:
$$s_1=x, \quad s_2=[{y}^{-1}xy,{x}^{-1}], \dots , 
 s_n=[ {y} ^{-1}s_{n-1} {y},s_{n-1}^{-1}],\dots ,$$

Recall the main result of \cite{BWW}.

\begin{Theorem}\label{th:BWW} (\cite{BWW})
A finite group  $G$ is solvable if and only if
$$(\exists \  n\in\BN) \ (\forall (x,y)\in G \times G)\ s_n(x,y)=1.$$
\end{Theorem}

The proof reduces to the following:

\begin{Theorem}\label{tt}  (\cite{BWW})
Let $G=PSL(2,\BF_q)$, $q>3$, or $G=Sz(2^{2m+1}).$
Then there exists a pair $(x,y)\in G\times G$ such that
$s_n(x,y)\ne 1$ for all $n\in \BN$.
\end{Theorem}

We want to give another proof of \thmref{tt} using the trace map and
other geometric considerations.

For technical reasons we will change notation and consider a
sequence $e_n(x,y)$ which differs from $s_n(x,y) $ only by replacing
$y$ with $y^{-1}.$ Since in \cite{BWW} the element $y$ was supposed
to be an involution, this does not matter. We define
$$e_1= x, \quad e_2=[ yxy^{-1},x^{-1}], \dots ,
 e_n=[ye_{n-1}y^{-1},e_{n-1}^{-1}],\dots ,$$
i.e. in this example
$$\varphi(x,y)=\varphi_y(x)=[yxy^{-1},x^{-1}]$$
(see \secref{psl}).

\medskip

{\bf Case of PSL}
As explained in Remark \ref{SL-PSL}, we can freely apply the results 
of \subsecref{2vm} obtained for $\tG =SL(2,q)$ to the case $G=PSL(2,q)$. 

We are going to compute the variety $\Phi$ of fixed points of the
corresponding trace map $\psi\colon\mathbb A^3\to\mathbb A^3$ (see
diagram (\ref{d4})). We maintain the notation of \subsecref{2vm}. In
particular, we denote $s=\tr (x),$\ $u=\tr (xy),$ \ $t=\tr (y),$ and
$r=u^2+s^2+t^2-ust.$ Then (see \cite[Lemma 5.2.4]{CMS}),
$$f_1(s,u,t)=2s^2+(\tr(yxy^{-1}x^{-1}))^2-s^2(\tr(yxy^{-1}x^{-1}))-2,$$
$$\tr[y,x]=r-2.$$

Direct computations 
give
\begin{equation}\label{1}
f_1(s,u,t)=2s^2+(r-2)^2-s^2(r-2)-2=s^2(4-r)+r^2-4r+2=(r-4)(r-s^2)+2,
\end{equation}
\begin{equation}\label{7}
f_2(s,u,t)=f_1(s,u,t)\cdot t+s(st-u)(r-4)-t(r-3).
\end{equation}

The variety $\Phi\subset \mathbb A^4$ is now defined by the following system:
\begin{equation}\label{w}
\Phi=\left\{ \begin{aligned} s=&(r-4)(r-s^2)+2,\\
                                    u=&st+s(st-u)(r-4)-t(r-3),\\
                                    r=&u^2+t^2+s^2-ust.
\end{aligned}\right\}\end{equation}
This curve contains a  trivial component $L_1$:
$$s=2,\qquad r=4,\qquad u=t.$$
To eliminate this component, we consider a curve $\tilde \Phi$ in
the space $\mathbb A^5$ with coordinates $(s,u,t,r,z)$  which is
isomorphic to $\Phi\setminus L_1$:
\begin{equation}\label{8}\tilde \Phi=\left.\begin{cases}  r=u^2+t^2+s^2-ust,\\
s=(r-4)(r-s^2)+2,\\
u=st+s(st-u)(r-4)-t(r-3),\\
z(r-4)=1. \end{cases}\right\}
\end{equation}

\begin{Lemma}\label{irr}
The plane curve $A\subset\mathbb A^2$ given by the equation
$(s-2)=(r-4)(r-s^2)$ is 
a smooth irreducible genus $1$ curve with two punctures.
\end{Lemma}

\begin{proof}
Assume the ground field is algebraically closed. Let $\tilde A$ be the closure of $A$ in the projective
space.
One can check 
that $\tilde A$
has no singular points.

As a  plane smooth curve,  $\tilde A$ is  irreducible. Moreover, it is a
double cover of $\mathbb P^1$ and by Hurwitz's formula has  genus 1.
\end{proof}

Magma computations 
show that the curve $\tilde
\Phi$ has two components: 
\begin{equation}\label{21}
W_1=\left.\begin{cases} z+t+s=0,\\ u-t-s+r-1=0,\\ ts-2t-2s+r=0,\\
tr-4t+sr-4s+1=0,\\ s^2r-4s^2+s-r^2+4r-2=0,\end{cases}\right\}\end{equation}
\begin{equation}\label{22}
W_2=\left.\begin{cases} z-t+s=0,\\ u-t+s-r+1=0,\\ ts-2t+2s-r=0,\\
tr-4t-sr+4s-1=0,\\ s^2r-4s^2+s-r^2+4r-2=0,\end{cases}\right\}\end{equation}
both defined over the ground field and isomorphic to
$A\setminus\{r=4,s=2\},$ i.e. to a genus 1 irreducible curve with 3
punctures.  Therefore both
$W_1$ and $W_2$ are absolutely irreducible.

From Theorem \ref{condition} it follows that if $q$ is big enough,
then there exists a pair $(x,y)\in PSL(2,q)\times
PSL(2,q)$ such that $x$ is a periodic point of the map
$\varphi_y.$

\begin{Remark}  Since $W_1, W_2$ are curves of genus 1 with 3
punctures, they contain $\BF_q$-points for all $q\ge
7.$ Since each fibre contains a rational curve with at most two
punctures, $q$ ``big enough''  means $q\ge 7$ in this example.
Small fields have been handled in a straightforward manner. 
\end{Remark}

{\bf Case of $Sz(2^n)$}

We keep the notation of \secref{sec:Suzuki}. We have to show that
the map $\varphi _y$ meets all the conditions of \thmref{ScSz}.
Condition \eqref{ce1} is obviously satisfied. Let us find an
invariant set $V$ of the map $\psi$ (see diagram (\ref{dd})). 
A direct computation of $f(0,b,k),$ \ $g(0,b,k)$ and
$h(0,b,k)$ for $x=T(0,b)D(k)w$ and  $y=T(0,1)$
gives 
$$k(0,b,k)=k^2\beta^{\frac{2s}{s+2}}=k^2(b+1)^{\frac{2s}{s+2}}\cdot
k^{\frac{(1+s)2s}{s+2}}=k^4(b+1)^{\frac{2s}{s+2}},$$
$$f(0,b,k)=0,$$
\begin{align*}
g(0,b,k)&=(\beta^{1-s}k^{-1})^{1+s}+(\beta+1/\beta)k^{-(1+s)}\\
&=k^{-(1+s)}(\beta^{1-s^2}+\beta+1/\beta)=k^{-(1+s)}\beta=b+1.\end{align*}
Thus for $b\ne 0,1$ the function $g$ has period 2.

After the second iteration, we get
$$f(f(0,b,k),g(0,b,k),h(0,b,k))=0,$$
$$g(f(0,b,k),g(0,b,k),h(0,b,k))=b,$$
$$h(f(0,b,k),g(0,b,k),h(0,b,k))=k^{16}(b+1)^{\frac{8s}{s+2}}b^{\frac{2s}{s+2}}.$$
Therefore, the set $V=\{x\in U_2: \pi_1(x)=0,$\ $\pi_2(x)=b\ne
0,1\}$ is invariant under the second iteration of $\varphi_y$ and
does not contain 1.

Theorem \ref{tt} is proved.

\subsection{Three-variable sequence}\label{un}

In this subsection we consider another sequence characterizing
solvable groups which was introduced in \cite{BGGKPP1},
\cite{BGGKPP2}:

$$u_0=x^{-2}y^{-1}x,\dots ,u_{n+1}=[xu_nx^{-1},yu_ny^{-1}],\dots$$

In the notation of \subsecref{3vm} we have
$$v(x,u,y)=[xux^{-1},yuy^{-1}], \ w(x,y)=x^{-2}y^{-1}x,$$
and $\ov a$ stands for the point $\ov
a=(a_1,a_2,a_3,a_{12},a_{13},a_{23},a_{123}) \in \BA^7.$

We need some additional notation:
$$a_{213}=\tr(yxu)= a_{12} a_3+ a_{13} a_2+a_{23} a_1-a_1a_2a_3,$$
$$b_{12}=\tr(x^{-1}y)=a_1a_2- a_{12},$$
$$b_{13}=\tr(x^{-1}u)=a_1a_3- a_{13},$$
$$b_{23}=\tr(y^{-1}u)=a_2a_3- a_{23},$$
$$b_{123}=\tr(x^{-1}yu)=a_1a_{23}- a_{123},$$
$$b_{213}=\tr(y^{-1}xu)=a_2a_{13}- a_{213},$$
$$c_{12}=\tr(xy^{2}) =a_{12}a_2- a_1,$$
$$c_{m12}=\tr(x^{-1}y^{2}) =b_{12}a_2-a_1,$$
$$d_{12}=\tr(x^{2}y) =a_{12}a_1- a_2,$$
$$d_{m12}=\tr(x^{-2}y) =b_{12}a_1- a_2,$$
$$g_{12}=\tr(xu^{2}) =a_{13}a_3- a_1,$$
$$f_{m23}=\tr(u^{2}y^{-1}) =b_{23}a_3- a_2,$$
$$p_1=\tr(ux^{-1}yuy^{-1}x)=a_3b_{12}b_{123}-b_{12}^{2}-b_{123}^{2}+2,$$
$$p_2=b_{23}p_1-b_{13}\{a_3b_{213}-b_{12}\}+ a_1b_{213}-b_{23},$$
$$p_3=b_{12}(a_2p_1-b_{13}b_{213}+d_{m12})-b_{213}a_{23}+c_{m12},$$
$$p_4=b_{12}^{2}+a_3^2+b_{123}^2-b_{12}a_3b_{123}-2,$$
$$p_5=b_{12}^{2}+a_3^2+b_{213}^2-b_{12}a_3b_{213}-2,$$
$$l_1(\ov a)= 2a_3^2+p_1^2-p_1a_3^2-2,$$
$$l_2(\ov a)=a_1l_1-b_{213}p_2+p_3,$$
$$\begin{aligned}
&l_3(\ov{a})=b_{213}(b_{13}p_1-(b_{123}f_{m23}-b_{12}b_{23}+b_{13}))-\\
&-b_{12}(p_1a_1-b_{123}b_{23}+c_{m12})+a_{13}b_{123}-d_{m12}.\end{aligned}$$

A direct computation shows that
\begin{equation}\label{ta1}
\tr([xux^{-1},yuy^{-1}])=l_1(\ov a),
\end{equation}
\begin{equation}\label{ta2}
\tr([xux^{-1},yuy^{-1}]x)=l_2(\ov a),
\end{equation}
\begin{equation}\label{ta3}
\tr([xux^{-1},yuy^{-1}]y)=l_3(\ov a).
\end{equation}

In the following paragraph we compute
$$\tr([xux^{-1},yuy^{-1}]xy)=l_4(\ov a):$$

$Y=b_{13}b_{213}-d_{m12}$,
$p_6=b_{12}^2+a_3^2+b_{123}^2-b_{12}a_3b_{123}-2$, $G=
b_{213}b_{12}a_3-b_{12}^2- b_{213}^2+2$, $U=a_2G-Y$,
$V=b_{213}a_{23}-c_{m12}$, $E=b_{12}U-V$, $Q=b_{213}a_1-b_{23}$,
$R=a_3b_{213}-b_{12}$, $H=b_{13}R-Q$, $D=b_{23}G-H$, $B=b_{123}D-E$,
$C=b_{12}(p_6-1)$, $A=a_2B-C$, $l_4=a_{12}l_1-A.$

Furthermore,
$$\tr(u_0)=\tr(x^{-2}y^{-1}x)=\tr(x^{-1}y^{-1})=a_{12},$$
$$\tr(u_0x)=\tr(x^{-2}y^{-1}x^2)=\tr(y)=a_{2},$$
$$\tr(u_0y)=\tr(x^{-2}y^{-1}xy)=\tr(x)\tr([x,y])-\tr(y^{-1}xy)=a_1(a_1^2+a_2^2+a_{12}^2-a_1a_2a_{12}-3),$$
$$\tr(u_0xy)=\tr(x^{-2}y^{-1}x^2y)=\tr([x^2,y])=(a_1-2)^2+a_2^2+d_{12}^2-(a_1-2)a_2d_{12}-2.$$
Therefore the variety $C=\Phi\bigcap W(w)$ is defined by equation
\eqref{hha1} and the following system of equations:
\begin{equation}\label{ta4}
l_1(\ov a)=a_3,
\end{equation}
\begin{equation}\label{ta5}
l_2(\ov a)=a_{13},
\end{equation}
\begin{equation}\label{ta6}
l_3(\ov a)=a_{23},
\end{equation}
\begin{equation}\label{ta7}
l_4(\ov a)=a_{123},
\end{equation}
\begin{equation}\label{ta8}
a_3=a_{12},
\end{equation}
\begin{equation}\label{ta9}
a_{13}=a_{2},
\end{equation}
\begin{equation}\label{ta10}
a_{23}=a_1(a_1^2+a_2^2+a_{12}^2-a_1a_2a_{12}-3),
\end{equation}
\begin{equation}\label{ta11}
a_{123}=(a_1-2)^2+a_2^2+d_{12}^2-(a_1-2)a_2d_{12}-2.
\end{equation}
Magma computations show that $C$ contains two components, $C_1$ and
$\Phi$: $C_1$ corresponds to the trivial solution $u_0=1$, $x=y^{-1}$, and
$\Phi$  is an irreducible curve intersecting the set $V$ (see Theorem \ref{s3vm} at a
finite number of points (at most $31$ as MAGMA computations give). 
Moreover,
this curve is a projection of the solution of the equation $u_0=u_1$
computed in \cite{BGGKPP2}.

\subsection{A new sequence}\label{y2}

In this subsection we produce a new sequence characterizing finite
solvable groups. It is a modification of the sequence $e_n$
considered in \subsecref{subsec:BWW}. We keep the notation of that
subsection.

Let $\theta_n(x,y)=s_n(x,y^2)$.
Denote $\theta(x,y)=\varphi(x,y^2),$ i.e.
$$\theta_y(x)=[y^2xy^{-2},x^{-1}].$$

\begin{Theorem} \label{th:new}
The map $\theta(x,y)\colon SL(2,q)\to SL(2,q)$ has nontrivial 
periodic points for all $q$. 
\end{Theorem}

\begin{proof}
For a pair $(x,y)\in SL(2,q)$, let $$s=\tr (x), \quad
t_1=\tr (y),\quad u_1=\tr (xy), \quad t=\tr (y^2)=t_1^2-2,\quad
\quad u=\tr (xy^2)=u_1t_1-s.$$

Consider the following maps:
$$\varkappa: \mathbb
A_{s,u_1,t_1}^3\longrightarrow \mathbb A_{s,u,t}^3,\qquad
\varkappa(s,u_1,t_1)=(s,u_1t_1-s,t_1^2-2);$$
$$\psi\colon\mathbb A_{s,u,t}^3\longrightarrow \mathbb
A_{s,u,t}^3,\qquad \psi(s,u,t)=(f_1(s,u,t),f_2(s,u,t),t),$$ where
the functions $f_1 $ and $f_2$ are defined in \eqref{1} and
\eqref{7}, respectively;
$$\psi_\theta\colon \mathbb A_{s,u_1,t_1}^3\to \mathbb A_{s,u_1,t_1}^3,$$
$$\psi_\theta(s,u_1,t_1)=(\tr\theta_y(x),\tr(\theta_y(x)\cdot
y),\tr y).$$

We obtain the following commutative diagram:
\begin{equation}\begin{CD}
{SL}(2)\times{SL}(2)@>(\theta ,\idd)>>SL(2)\times{SL}(2)\\
@V\pi VV @V\pi VV \\
\mathbb A_{s,u_1,t_1}^3@> \psi_\theta  >>\mathbb A_{s,u_1,t_1}^3\\
@V\varkappa VV @V\varkappa VV \\
\mathbb A_{s,u,t}^3@>   \psi >>\mathbb A_{s,u,t}^3
\end{CD}
\label{3store}\end{equation}

As shown above, the variety $\Phi$ of fixed points of $\psi$ has
three irreducible $\mathbb F_q$-components $L_1$, $W_1$, $W_2$, all
absolutely irreducible for any $q.$

\begin{lemma}\label{lem:2}
The curve $Z_2:=\varkappa^{-1}(W_2)$ is absolutely irreducible.
\end{lemma}

\begin{proof}
Consider the curve $\ov B$ defined in $\BP^3$ with homogeneous
coordinates $(\tilde s:\tilde r:\tilde t:\tilde w)$ by the
equations:
\begin{equation}\label{27}
 \tilde s\tilde t-2\tilde t\tilde w+2\tilde s\tilde w-
\tilde r\tilde w=0,\end{equation}
\begin{equation}\label{28}\tilde t\tilde r-4\tilde t\tilde w-\tilde s\tilde r+4\tilde s\tilde w-
\tilde w^2=0,\end{equation}
\begin{equation}\label{29}(\tilde s-2\tilde w)\tilde w^2=(\tilde r\tilde w-\tilde s^2)(\tilde r-4
\tilde w).\end{equation}
Since equations \eqref{22} are linear in $u$ and $z$, the curve $\ov
B$ is isomorphic (or at least birational and one-to-one) to the
projective closure of $W_2$.

The curve $\ov C\subset \BP^4$, isomorphic  (or at least  birational
and one-to-one) to the closure of $Z_2$, can be defined in $\BP^4$
with coordinates $(\tilde t_1:\tilde s:\tilde r:\tilde t:\tilde w)$
by the same system \eqref{27}, \eqref{28}, \eqref{29}, together with
the additional equation
\begin{equation}\label{30}
  \tilde t_1^2=\tilde w(\tilde t+2\tilde w).
\end{equation}
The projection $\tau \colon \ov C\to\ov B,$
$$\tau(\tilde t_1:\tilde s:\tilde r:\tilde t:\tilde w)=(\tilde s:\tilde r:\tilde t:\tilde w),$$
is a morphism  which represents
$\ov C$ as  a ramified double
cover of $\ov B$ (this can be checked by a direct computation). 
Since  $\ov B$ is absolutely irreducible, so is
$\ov C.$
\end{proof}

From diagram \eqref{3store} it follows that at least the second
iteration of $\psi_{\theta}$ has a nontrivial absolutely irreducible
component in the variety of its fixed points. Formula \eqref{30}
shows that $\ov C$ is a double cover of $\ov B$ with at most three
ramification points (all  at infinity). It follows that the genus is
at most $2.$ Since $B$ has  $3$ punctures and over at least one of
them  $\ov C$ is ramified,  $C$  has at most 5 punctures. Therefore
for $q\ge 13$ there are points in $Z_2$  rational over $\BF_q.$

The case $q<13$ was checked by straightforward computations.
\end{proof}

\begin{cor}\label{newseq}
A finite group $G$ is solvable if and only if
$$(\exists   n\in\BN )\ (\forall (x,y)\in G \times G)\ \theta_n(x,y)=1.$$
\end{cor}

\begin{proof}
We argue as in the proof of Theorem \ref{th:BWW}. 
Theorem \ref{th:new} settles the $PSL(2,q)$ case. In the case
$Sz(2^n)$ no new proof is needed because $T(0,1)=T(1,1)^2.$ Periodic
points of $\vp_y$ with $y=T(0,1)$ are periodic points of
$\theta_{y_1}$ with $y_1=T(1,1).$ The case $G=PSL(3,3)$ is straightforward: 
for the matrices 
$$
x=\left(\begin{matrix} 
2 & 0 & 0 \\
0 & 0 & 1 \\
0 & 1 & 2
\end{matrix}
\right),
\quad 
y=\left(\begin{matrix} 
0 & 2 & 2 \\
1 & 2 & 1 \\
0 & 2 & 0
\end{matrix}
\right)
$$
we have $s_1(x,y)=s_4(x,y)$. 
\end{proof}

\begin{Remark}
The proof of \cite{BWW} does not work for the sequence from Theorem
\ref{newseq}. It is proved in \cite{BWW} that for $$y_0=\left(\begin{matrix} 0 & -1\\
1& 0\end{matrix}\right)$$ there exists a periodic point of
$\varphi_{y_0}$ in $SL(2,q)$ for every $q$. But $y_0\ne z^2$
in $SL(2,q)$ if 2 is not a square in $\mathbb F_q$.
\end{Remark}

\begin{Remark} We believe that the statement of Theorem
\ref{newseq} remains true if one takes $y^n$, with any $n\ge2$,
instead of $y^2$ (at least for even $n$) but this requires more subtle analysis.
\end{Remark}


\subsection{Commutator}\label{Com}

In the following example we want to show how useful the trace 
method can be. We present a very simple proof of the following
statement (which is a very special case of a theorem of Borel
\cite{Bo}, see also \cite{La}):

\begin{Example}\label{comm}
Let $G=SL(2,q)$. Then the map $F\colon G\times G\to G$ defined
by $F(x,y)=[x,y]$ is a dominant morphism of the underlying algebraic
$\BF_q$-varieties.
\end{Example}




\begin{proof}
In the notation of \subsecref{2vm}, consider the corresponding map
$\psi\colon\BA^3_{s,u,t}\to\BA^3$: if $\tr(x)=s, \tr(y)=t,
\tr(xy)=u$, then
$$\psi(s,u,t)=(f_1(s,u,t), f_2(s,u,t), t).$$
Here  $f_1(s,u,t)=\tr(F(x,y))=s^2+t^2+u^2-ust-2, f_2(s,u,t)=t.$

Let $z\in G$, and suppose that $a=\tr(z)\ne\pm2$. We want to show  that 
there exist $x,y\in G$ with $[x,y]=z$. 

For any $t\in\BF_q$ consider the inverse image
$\G_{a,t,t}:=\psi^{-1}(a,t,t)\subset \BA^3_{s,u,t}$. We have 
$$\G_{a,t,t}=\{(s,u,t)\in \BA^3: \ s^2+t^2+u^2-ust-2-a=
0\}.$$

For a fixed value $t_0\ne \pm 2,$  this is a quadratic equation 
in $(s,u)$ which 
has a solution $(s_0,u_0) $ over every finite field 
(cf. the proof of Theorem \ref{rational}).
Thus we have a point $Q:=(s_0,u_0,t_0)\in\G_{a,t_0,t_0}$. 

By \thmref{rational}, $\pi^{-1}(Q)\ne\emptyset$, so take 
$(x,y)\in \pi^{-1}(Q)$. We have $\tr(F(x,y))=a=\tr (z)$.
If  $a\ne\pm 2$ 
(i.e. $z$ is semisimple), $F(x,y)$ is conjugate to $z$, i.e. 
$[x,y]=wzw^{-1}$. We get $[w^{-1}xw,w^{-1}yw]=z$, as required.   
\end{proof}

The map $F\colon G\times G\to G$ provides a dynamical system on $SL(2,q)\times SL(2,q)$
with $\tilde\phi(x,y)=([x,y],y)$. It corresponds to the Engel sequence
$e_1=[x,y]$,\dots , $e_{n+1}=[e_n,y],\dots$

Let us show that this dynamical system has nontrivial periodic points  
for every $q$. The cases $q=2,3$ are treated by a direct computation, so assume $q>3$.   
In view of \thmref{rational}, 
it is sufficient to find a fixed point  of
 the trace map 
$$\psi(s,u,t)=(s^2+t^2+u^2-ust-2, t,t)$$
with $s^2\ne 4$, $t^2\ne 4$.  
The point $(s,t,t)$ is fixed if
$s= s^2+2t^2-st^2-2.$
If $q=2^n$, then any pair $(s=1+t^2,t)$ is a needed  solution of this equation.  
If $q\ne 2^n, $
then for a fixed $t$ we get 
$s_1=2$ (forbidden), $s_2=t^2-1$.   
Thus, any pair $(t^2-1, t)$,  $t^2\ne -1,3,4$ provides a needed fixed point.

\section{Possible generalizations} \label{final}


Here we present some more general problems arising from
concrete calculations of the preceding sections. In Subsection
\ref{subsec:rpds} we consider AG systems introduced in \secref{intro} making this 
notion more precise. In particular, we want to distinguish between the cases 
when the underlying geometric object is defined over a global field or its 
ring of integers. We define residually periodic dynamical systems, propose 
some relevant conjectures and give several examples. In \subsecref{subsec:vds}
we discuss in more detail verbal dynamical systems defined in the introduction. 
By combining the notions of AG dynamical system and verbal dynamical system, 
we define systems carrying both structures.


\subsection{Residually periodic dynamical systems}
\label{subsec:rpds}

We start with AG dynamical systems. 

Let $K$ be a global field, and let $\SO$ stand for the ring of
integers in $K.$

\begin{Definition}\label{D}
A triple $D=(X, V, \vp)$ is called a $K$-dynamical system if
\begin{itemize}
\item $X$ is  an algebraic $K$-variety;
\item $\vp\colon X\to X$ is a dominant $K$-morphism;
\item $V\subset X(K)$ is a $\vp$-invariant subset.
\end{itemize}
\end{Definition}


\begin {Definition}\label{sD}
A triple $\SD = (\SX, \SV, \Phi)$ is called an $\SO$-dynamical
system if
\begin{itemize}
\item $\SX$ is  an $\SO$-scheme of finite type;
\item  $\Phi\colon\SX\to \SX$ is a dominant $\SO$-morphism;
\item $\SV\subset \SX(\SO)$ is a $\Phi$-invariant subset.
\end{itemize}
\end{Definition}

We say that an $\SO$-dynamical system $\SD = (\SX, \SV, \Phi)$ is an
integral model of $D=(X, V, \vp)$ if
\begin{itemize}
\item $\SX\times_{\SO}K=X;$
\item the restriction of $\Phi$ to the generic fibre coincides with $\vp$;
\item  $R(\SV)=V,$ where $R\colon\SX\to X$ is the restriction to the generic fibre.
\end{itemize}

Consider a $K$-dynamical system $D=(X, V, \vp)$ and its integral
model $\SD = (\SX, \SV, \Phi).$ For a place $p$ of $K$ let
\begin{itemize}
\item $ \kap _{p}$ be the residue field
of $p;$
\item  $\SX_p$  the special fibre
of $\SX$ at $p; $
\item $R_p\colon\SX \to\SX_p$ the reduction map (restriction to the special fibre);
\item $\vp_p\colon\SX_p\to\SX_p$ the reduction of $\Phi$ viewed as a morphism of
$\kap _{p}$-schemes;
\item $X_p=\SX_p(\kap _{p})$ the set of rational points;
\item $V_p=R_p(\SV)\subset X_p$ the reduction of $\SV.$
\end{itemize}

Assume that for all but finitely many places $p$ the scheme  $\SX_p$
is integral. One can deduce from \cite[9.6.1(ii)]{Gr} that for all
but finitely many $p$'s the reduced morphism $\vp_p$ is dominant.
Let $z\in X_p\setminus V_p $  be a periodic point of $\vp_p.$ Let
$\ell(z)$ be the number of distinct points in the orbit of $z.$ Set
$\ell_p:=\min\{\ell(z)\}$ where the minimum is taken over all $z$'s
as above. If there are no periodic points in $X_p\setminus V_p, $ we
set $\ell_p=\infty.$ Let $M$ denote the collection of primes $p$
such that $\ell_p=\infty.$  Let $N=\{\ell_p\}_{p\not\in M}.$

\begin{Definition} \label{resperv}
With the above notation, we say that a $K$-dynamical system $D=(X,
V,\vp)$ or an $\SO$-dynamical system $\SD = (\SX, \SV, \Phi)$ is
{\em residually aperiodic} if the set $M$ is infinite, {\em
residually periodic} if $M$ is finite, and {\em strongly residually
periodic} if the sets $M$ and $N$ are both finite.
\end{Definition}

For example, in \subsecref{subsec:BWW} for a map
$\psi\colon\BA^3_{\BZ}\to \BA^3_{\BZ}$ we had $\SX=\BA^3,$
$\kap_p=\BF_p,$  $\SV=\{(\pm 2,\pm t,t)\},$ $N=\{1\}$ and
$M=\emptyset$.

We believe that the following special case is particularly
interesting. Let $\SV\subset \SX(\SO)$ be the set of all preperiodic
integer points (i.e. points having a finite orbit). Let
$V_p=R_p(\SV)\subset X_p$ be its reduction mod $p.$ Residual
periodicity of $\SD = (\SX, \SV, \Phi)$ means that $\vp_p$ has
periodic points outside $V_p$ for all but finitely many $p$'s. In
simple words, we are looking for periodic points of $\vp_p$ not
coming from preperiodic integer points of $\Phi.$ Note that
according to \cite{Si2}, cycles coming from a fixed nonperiodic
integer point cannot be too short (their length, as a function of
the cardinality of the residue field, tends to infinity). Thus our
approach to studying cycles of reduced systems is, in a sense,
complementary to \cite{Si2}.
\medskip

As mentioned in the introduction, there may be different reasons 
for a dynamical system to be residually periodic. For higher-dimensional 
systems one can look for geometric conditions. The next notion captures the 
phenomenon of extra coordinates, or more generally invariant functions, as in 
Example \ref{fibered}. 

\begin{Definition} \label{def:fibre}
We say that  a dynamical system  $D=(X, V,\vp) $ is of {\em fibred}
type if there exists a regular function $f$ on $X$ such that $f\circ
\vp^{(n)}=f$ for some iteration $\vp^{(n)}$ of $\vp.$
\end{Definition}

\begin{question} \label{prob:fibre}  Assume
that  a dynamical system  $D=(X, V,\vp) $ is of {\em fibred} type.
Assume that the endomorphism $\vp $ is not birational. Under what conditions 
on $\vp$ is $D$ strongly  residually periodic?
\end{question}

Question \ref{prob:fibre} is essentially higher-dimensional. In
one-dimensional situations the main role, of course, belongs to
arithmetic. To get a better feeling of the problem, it is useful to
consider one-dimensional examples which are, in a sense, opposite to
Example \ref{splitfield} from the introduction.

\begin{Example} \label{fritz}
Let $\CT={\mathbb G}_{{\text{\rm{m}}},\BZ }=\Spec (\BZ
[x,y]/(xy-1))$ be the trivial one-dimensional torus. Fix a
positive integer $d$, and let $\Phi\colon \CT\to \CT$ denote the
power map: $t\to t^d.$ The set of integer points $\SR=\CT (\BZ)$
consists of two points, $1$ and $-1,$ both fixed under $\Phi$
(i.e. periodic with period one). We choose the forbidden set
$\SV=\SR.$ If $d=2,$ then $\ell_p=\infty$ for every Fermat prime
$p=2^{m}+1$. Thus the system is residually periodic or aperiodic
depending on whether there are finitely or infinitely many Fermat
primes. Assume now that $d$ is odd.
\end{Example}

\begin{Proposition} \label{prop:torus}
The dynamical system $( \CT, \SR, \Phi)$ of \exampref{fritz} is
residually periodic but is not strongly residually periodic.
\end{Proposition}

\begin{proof}

We have $X_p=\BF^*_p,$ and for any $t\in\BF^*_p$ we have
$\vp^{(n)}(t)=t^{d^n}$. Assume $(p,d)=1$. We are looking for
$t\ne\pm 1$ such that
\begin{equation} \label{period}
t^{d^n-1}\equiv 1\pmod p.
\end{equation}
To find such a $t$ with minimal possible $n$, let us first
introduce some notation. For any prime $\ell$ such that $(d,\ell
)=1$ denote by $s_\ell$ the order of $d$ in $\BF^*_\ell$. Denote
by $Q(p)=\{q_i\}$ the set of all odd primes appearing in the prime
decomposition of $p-1$ and coprime to $d$. Set $a(p):=\min_{q\in
Q}s_q$. If $p\equiv 1 \pmod 4$, we have $\ell_p\le 2$. 
We claim that for $p\equiv -1\pmod 4$ we have $\ell_p=a(p)$. 
Indeed, suppose that the
minimum is achieved at some $q\in Q$, so $d^{s_q}-1=qm$ for some
integer $m$. If $g$ is a primitive element of $\BF_p$, one can
take $t=g^{(p-1)/q}$ and $n=s_q$ to satisfy (\ref{period}). On the
other hand, if $n<s_q$, then by the definition of $s_q$ we have
$n<s_\ell$ for all $\ell\in Q$, and hence for all such $\ell$ we
have
$$
d^n\not\equiv 1 \pmod\ell .
$$
The above also holds for all $\ell$ dividing $d$, so we conclude
that $(d^n-1,p-1)=2$. If (\ref{period}) holds for some $t$, then
the order of $t$ must divide both $d^n-1$ and $p-1$, hence it is
equal to 2. Thus $t=-1$ and belongs to the reduction of the
forbidden set $\SR$. We conclude that (\ref{period}) does not hold
for any $n<s_q$. This means that $s_q=a(p)$ is the minimal
possible length of the orbit of $\varphi_p$, i.e. $\ell_p=a(p)$.


To finish the proof of the proposition, it is enough to establish the following
simple lemma (we thank Z.~Rudnick for an  elementary proof):
\begin{Lemma} \label{sieve}
The set $A=\{a(p)\}$, where $p$ runs over all prime numbers congruent to 
$-1$ modulo $4$, is
infinite.
\end{Lemma}
\noindent {\it Proof of the Lemma.} Assume the contrary:
\begin{equation} \label{eq:orders}
A=\{s_{q_1}, \dots ,s_{q_t}\}.
\end{equation}

To get a contradiction, we wish to find $p\equiv -1\pmod 4$ with $a(p)\notin A$.

First note that there are at most finitely primes $q$ with a given
value of $s_q$, and denote by $B$ the set of all $q$ such that
$s_q\in A$. It follows that $B$ is finite. Thus we have to find a
prime $p$ such that $p-1$ is not divisible by any $q\in B$.
We want to find a prime number $p$ satisfying the system of
congruences
$$
\begin{aligned}
x\equiv -1 \pmod 4,\\
x\equiv -1 \pmod {q}
\end{aligned}
$$
for all $q\in B$. By the Chinese Remainder Theorem, the solutions
of this system form an arithmetic progression. By Dirichlet's
Prime Number Theorem, this progression contains infinitely many
primes. If now $p$ is such a prime, we have $p\not\equiv 1 \pmod
q$ for any $q\in B$. Thus the order of $d$ in $\BF_p^*$ is not
equal to any of $s_{q_i}$'s, and so $p\notin A$, contradiction.

This finishes the proof of the lemma and hence of Proposition
\ref{prop:torus}.
\end{proof}

\begin{Example} \label{elliptic}
Let now $E$ be a CM elliptic curve defined over $\mathbb Q$ by the
equation $y^2=x^3-x$, and let $\CE$ denote its minimal Weierstrass
model. Let $\Phi\colon\CE\to\CE$ be the
multiplication-by-$d$ map ($d$ stands for a positive odd integer).
There are four 2-torsion points: (0,0), (1,0), (-1,0) and $\infty$,
all belonging to $\CE (\BZ )$.
Denote this collection by $\SV.$
If $p\equiv -1 \pmod 4$, the
reduction of $E$ is supersingular, i.e. $\vert E(\BF_p)\vert =p+1$.
We can now denote by $b(p)$ the smallest prime factor of the number
$\vert E(\BF_p)\vert /4$ and by the argument similar to that of the
previous example show that the set $B=\{b(p)\}$, where $p$ runs over
all $p\equiv -1\pmod 4$, is infinite. This leads to
\begin{Proposition} \label{prop:ell}
The dynamical system $\SD=( \CE,\SV,\Phi)$ is residually periodic
but is not strongly residually periodic. \qed
\end{Proposition}
\end{Example}

The interested reader is invited to complete the details of the
proof as well as to develop more examples of arithmetical interest.

To go beyond CM elliptic curves, one needs more efforts. A natural 
question to ask is the following one: 

\begin{question} \label{ellsieve}
Let $E$ be an elliptic curve over $\mathbb Q$, and let $D$ denote
the order of its rational torsion. For each place $p$ of good
reduction, denote by $c(p)$ the smallest prime divisor of the number
$\vert E(\BF_p)\vert /D$. Can the set $C=\{c(p)\}$, where $p$ runs
over all places of good reduction of $E$, be finite? Can the
system $(\CE, \CE(\BQ)_{tors}, \Phi)$ be strongly residually
periodic?
\end{question}

At the first glance, the conjectures by Lang--Trotter
\cite{LT} and Koblitz \cite{Ko}, predicting (for most elliptic curves) infinitely many $p$'s
with $\vert E(\BF_p)\vert$ of prime order, give little hope to
find an example of an elliptic curve such that the dynamical system
defined by the multiplication-by-$d$ map is strongly  residually periodic.
However, the following example (due to N.~Jones) prevents from making 
hasty conclusions. Consider the curve $E_0$ 
given over $\mathbb Q$ by the Weierstrass equation 
$$
y^2=x^3+75x+125.
$$
N.~Jones proved that although $E_0$ has no rational torsion, the order of $E_0(\BF_p)$ is divisible 
either by 2 or by 3 for all $p>5$. The curve $E_0$ is of Mordell--Weil rank 1, so the multiplication-by-$d$ 
map $\Phi$ induces a nontrivial dynamical system $\SD=(\CE_0, \infty, \Phi)$. Taking, say, $d=7$, we 
conclude that $\SD$ is strongly residually periodic in the strongest possible sense: it has no periodic 
points but the residual system $\SD_p$ has a fixed point for all $p>5$ (compare with Example \ref{splitfield}).     
   
On the other hand, N.~Jones proved (unconditionally on Koblitz's conjectures) that 
for a ``typical'' elliptic curve $E$ over $\mathbb Q$ an analogue of Lemma
\ref{sieve} indeed holds which implies that the dynamical system $\SD$
is not strongly residually periodic for such an $E$, i.e. typically the answer to Question \ref{ellsieve} 
is negative. See the Appendix for more details.



\subsection{Verbal dynamical systems on group schemes}
\label{subsec:vds}

We view the calculations of \secref{psl} as a first step in
attacking one of the most important conceptual questions left open
after discovery of two-variable sequences characterizing finite
solvable groups: for a sequence of words in the free group on two
generators, to what extent the property to characterize the class of
finite solvable groups is a property of general position, and what
type of the dynamic behaviour is typical? Questions of such
``nonbinary'' type, which do not admit an answer of type ``yes-no'',
have been considered by many mathematicians, from Poincar\'e to
Arnold, as the most interesting ones. Dynamics of word maps in free
group, in spirit of \cite{LP}, \cite{La}, \cite{Sh}, \cite{LS},
\cite{GS}, led to a breakthrough in some classical problems of the
theory of finite groups, and it may happen to play a crucial role in
the above mentioned problem as well.
 Namely, a possible goal is to
prove (or disprove) that for a sufficiently wide class of sequences
the property to characterize the class of finite solvable groups
holds in ``general position'' and is determined by its dynamics in
the free group. In what follows $F_r$ stands for the free group on
$r$ generators.

\begin{question} \label{prob:solv}
Suppose that a sequence $\overrightarrow{u} = u_1, u_2, \dots,
u_n, \dots$ of elements of $F_2$ satisfies the following conditions:

(i) $u_n(a,1)=u_n(1,g)=1$ for all sufficiently big $n$, every group
$G$, and all elements $a,g\in G$;

(ii) if $G$ is any group and $a,g$ are elements of $G$ such that
$u_n(a,g)=1$, then for every $m>n$ we have $u_m(a,g)=1$;

(iii) no element of $\overrightarrow{u}$ equals $1$ in $F_2$;

(iv) there exists $N$ such that for all $n>N$ the word $u_n(x,y)$
belongs to the $n$-th derived subgroup $F_2^{(n)}$ of $F_2$.


Is it true that if a finite group $G$ satisfies an identity
$u_n(x,y)\equiv 1$ for some $n$, then it is solvable?
\end{question}

In connection with Question \ref{prob:solv}, it is natural to pose

\begin{problem} \label{prob:words}
To describe the words in $F_2$ satisfying conditions (i)--(iv) of
Question $\ref{prob:solv}$.
\end{problem}

Extensive MAGMA computations show strong numerical evidence of a
positive answer to Question \ref{prob:solv}, at least for the class
of sequences $\overrightarrow{u}$ studied in \cite{Ri}: $u_0:=f$,\dots ,
$u_n:=[gu_ng^{-1}, hu_nh^{-1}],\dots $, where $f,g,h$ stand for some words
from $F_2$.

\medskip

One can put Question \ref{prob:solv} into somewhat more general
context. Towards this end, we suggest to combine the notions of 
verbal and AG dynamical systems defined in \secref{intro}. For simplicity 
we restrict ourselves to considering $\BZ$-dynamical systems. 

\begin{Definition} \label{verbal}
A verbal dynamical $\BZ$-system consists of the following setup:
\begin{itemize}
\item positive integers $r,s$;
\item an $r$-tuple $\CW =(w_1,\dots ,w_r)$ of words in the free group $F_{r+s}$; 
\item an $r$-tuple $\CJ =(f_1,\dots ,f_s)$ of words in the free group $F_s$ (optional); 
\item a group scheme $\CG$ of finite type over $\BZ$;
\item a set $I\subset \CG^{r+s} (\BZ )$.
\end{itemize}
The following assumptions are to be satisfied.

(i)  Let  $D_{\CW}\colon\CG^{r+s}\to\CG^{r+s}$ be a morphism of $\BZ$-schemes  
defined on the group $G=\CG ^{r+s}(A)$ of $A$-points of $\CG^{r+s}$ for every $\BZ$-algebra $A$ 
by the formula 
$$(g_1,\dots, g_s, v_1,\dots, v_r)\mapsto (g_1,\dots, g_s, w_1(g_1, \dots,g_s, v_1,\dots,v_r),\dots, w_r(g_1, \dots,g_s, v_1, \dots,v_r).$$
Then we assume that $D_{\CW}$ is {\em dominant}.

(ii) The set $I$ is {\em invariant}, i.e. $D_{\CW }(I)\subset I$.
\end{Definition}

Our earlier considerations (cf. Examples \ref{s=2,r=1} and \ref{s=r=1})  
naturally fit into this setting if $\CG$ is a semisimple Chevalley group scheme 
over $\BZ$ (e.g., $\CG =SL(2,\BZ )$ as in the present paper). 
Indeed, in that case by a theorem of Borel
(\cite{Bo}, see also \cite{La}), the morphism $D_{\CW }$ is dominant, and we arrive at a verbal
dynamical $\BZ$-system in the sense of Definition \ref{verbal}.  
Remark \ref{rem:int} shows that the
dynamical systems on $SL(2,p)$ 
relevant for our original problem, can be viewed as special fibres
of the original verbal $\BZ$-system.

\begin{Remark}
It would be interesting to formulate a word-theoretic condition on
$\CW$ guaranteeing that for any Chevalley group scheme $\CG$ the
morphism $D_{\CW }$ is dominant.
\end{Remark}

In connection with Question \ref{prob:fibre} one can pose

\begin{problem} \label{prob:verbaper}
Given a verbal dynamical $\BZ$-system, that is not of fibred type, find
conditions under which it is (strongly) residually periodic.
\end{problem}


In particular, it would be interesting to consider the system
from Section \ref{2vm} given by the map $\vp_y\colon SL(2,\BZ )\to SL(2,\BZ )$ 
($y$ fixed) with $I=\{1\}$. This system has an invariant
rational function, but it is not regular.
It was proven in \cite{BWW} that for
$$y=\left(\begin{matrix} 0 & -1\\
1& 0\end{matrix}\right)$$ it is residually periodic. On the other
hand, our numerical experiments give some evidence that it is not
strongly residually periodic.

\medskip

We believe that verbal dynamical systems deserve more thorough
study. To the best of our knowledge, most arithmetically interesting
questions, in spirit of the monograph \cite{Si1} (boundedness of
periods, distributions of periods in reductions, various
local-global problems), are widely open (or even not yet posed at
all).

\medskip

\noindent {\it Acknowledgements}. Bandman and Kunyavski\u\i \ were
supported in part by the Ministry of Absorption (Israel) and the
Israel Academy of Sciences grant 1178/06.  A substantial part of
this work was done during the visits of Bandman and Kunyavski\u\i \
to MPIM (Bonn) in 2007 and 2009 and the visits of Grunewald to Bar-Ilan
University in 2008 and the Hebrew University of Jerusalem in 2009, and 
discussed by all the coauthors during the
international workshops
(2007, 2009) hosted by the Heinrich-Heine-Universit\"at (D\"usseldorf),
and the Oberwolfach meeting  ``Profinite and Geometric Group
Theory'' in 2008 (the visits were supported in part by the Minerva
Foundation through the Emmy Noether Research Institute of
Mathematics). The appendix to the paper arose from questions posed 
by Kunyavski\u\i \ to Jones when they participated in the program  
``Diophantine equations'' organized by the Hausdorff Research Institute for  
Mathematics (Bonn) in 2009. The support of all above institutions is highly
appreciated.

We are very grateful to V.~Berkovich, F.~Campana,
J.-L.~Colliot-Th\'el\`ene, N.~Fakhruddin, M.~Leyenson, Z.~Rudnick, 
M.~Tyomkin, and S.~Vishkautsan for useful discussions and correspondence.


\newpage

\setcounter{equation}{0}
\setcounter{section}{0}

\renewcommand{\theequation}{A-\arabic{equation}}
\renewcommand{\thesection}{A\arabic{section}}

\theoremstyle{plain}

\newtheorem{atheorem}{Theorem A\!}
\newtheorem{alemma}[atheorem]{Lemma A\!}
\newtheorem{acorollary}[atheorem]{Corollary A\!}
\newtheorem{asublemma}[atheorem]{Sublemma A\!}
\newtheorem{aproposition}[atheorem]{Proposition A\!}
\newtheorem{aconjecture}[atheorem]{Conjecture A\!}

\theoremstyle{definition}

\newtheorem{adefinition}[atheorem]{Definition A\!}
\newtheorem{aexample}[atheorem]{Example A\!}
\newtheorem{aremark}[atheorem]{Remark A\!}

\newtheorem{adesired result}[atheorem]{Desired result A\!}

\newtheorem{aquestion}[atheorem]{Question A\!}
\newtheorem{anotation}[atheorem]{Notation A\!}
\newtheorem{aproblem}[atheorem]{Problem A\!}
\newtheorem{aassumption}[atheorem]{Assumption A\!}
\newtheorem{aexercise}[atheorem]{Exercise A\!}

\newcommand{\begd}{\begin{displaystyle}}
\newcommand{\gl}{\lambda}
\newcommand{\gL}{\Lambda}
\newcommand{\gge}{\epsilon}
\newcommand{\gG}{\Gamma}
\newcommand{\ga}{\alpha}
\newcommand{\gb}{\beta}
\newcommand{\gd}{\delta}
\newcommand{\gD}{\Delta}
\newcommand{\gs}{\sigma}
\newcommand{\mbq}{\mathbb{Q}}
\newcommand{\mbr}{\mathbb{R}}
\newcommand{\mbz}{\mathbb{Z}}
\newcommand{\mbc}{\mathbb{C}}
\newcommand{\mbn}{\mathbb{N}}
\newcommand{\mbp}{\mathbb{P}}
\newcommand{\mbf}{\mathbb{F}}
\newcommand{\mbe}{\mathbb{E}}
\newcommand{\lcm}{\text{lcm}\,}
\newcommand{\gal}{\textrm{Gal}\,}
\newcommand{\mf}[1]{\mathfrak{#1}}
\newcommand{\ol}[1]{\overline{#1}}
\newcommand{\mc}[1]{\mathcal{#1}}
\newcommand{\nequiv}{\equiv\hspace{-.13in}/\;}



\part*{Appendix. Primes $p$ for which $\# E(\mbf _p)$ has only large prime factors}

\medskip

\centerline{\sc{Nathan Jones}}

\bigskip



\section{Introduction} \label{introduction}

Let $E$ be an elliptic curve over $\mbq$ of conductor $N_E$.  For each prime $p$ of good reduction for $E$, consider the group $E(\mbf_p)$ of $\mbf_p$-points of $E$.  In $1988$, Koblitz \cite{koblitz} conjectured a precise asymptotic formula for the number of good primes $p$ up to $x$ for which $\#E(\mbf_p)$ is prime.
\begin{aconjecture} \label{koblitzconjecture}
There exists a precise constant $\mf{S}_E \geq 0$ so that
\[
\# \{ p \leq x : p \nmid N_E \text{ and } \#E(\mbf_p) \text{ is prime} \} = \mf{S}_E \cdot \frac{x}{\log^2 x} + o\left( \frac{x}{\log^2 x} \right),
\]
as $x \longrightarrow \infty$.
\end{aconjecture}
In particular, provided the constant $\mf{S}_E > 0$, Conjecture A\ref{koblitzconjecture} implies that there are infinitely many primes $p$ for which $\#E(\mbf_p)$ is prime.  In case $\mf{S}_E = 0$, one can prove (as a consequence of the Chebotarev density theorem) that $\#E(\mbf_p)$ is prime for only finitely many primes $p$.

Based on the precise form of the predicted constant $\mf{S}_E$, Koblitz further noted that $\mf{S}_E$ is positive if and only if every other elliptic curve $E'$ over $\mbq$ which is $\mbq$-isogenous to $E$ has no non-trivial rational torsion:
\begin{equation} \label{positivityofse}
\mf{S}_E > 0 \Longleftrightarrow \left( E' \sim_\mbq E \Rightarrow E'(\mbq)_{\textrm{tors}} = \{ \mc{O}_{E'} \} \right).
\end{equation}

However, because of a technical error in the underlying heuristic, the constant $\mf{S}_E$ appearing in the original conjecture is incorrect.  A refined conjecture, which in particular corrects $\mf{S}_E$, has recently been given by D. Zywina \cite{zywina}. 
 In the interest of consistency, let us henceforth understand the symbol $\mf{S}_E$ appearing in Conjecture A\ref{koblitzconjecture} to refer to the corrected constant $C_{E,1}$ appearing in \cite[Conjecture $1.2$]{zywina} (we will describe this constant in more detail in Section \ref{theconstant}).  Having thus replaced $\mf{S}_E$, the interpretation \eqref{positivityofse} of exactly when $\mf{S}_E$ is positive is no longer valid.  We will show this in Section \ref{positivity} by exhibiting an elliptic curve $E$ over $\mbq$ for which the right-hand side of \eqref{positivityofse} is true, but for which $\mf{S}_E = 0$ nevertheless.  

In spite of various partial results (see for instance \cite{bacoda} and the references therein), Conjecture A\ref{koblitzconjecture} is still open.  Our goal is to prove the following theorem, wherein we relax ``is prime'' to ``has only large prime factors.''  Let us denote by
\[
c_E(p) := \min \{ \ell \text{ prime}: \ell \mid \#E(\mbf_p) \}
\]
the smallest prime divisor of $\#E(\mbf_p)$.
\begin{atheorem} \label{maintheorem}
Suppose that 
\begin{equation*} \label{positivitycondition}
\mf{S}_E > 0,
\end{equation*}
where $\mf{S}_E$ is the constant appearing in Conjecture A\ref{koblitzconjecture}.  Then the set
\[
\{ c_E(p) : p \nmid N_E \}
\]
is unbounded.  
\end{atheorem}
In other words, Theorem A\ref{maintheorem} asserts that, for each $x > 0$, there exists a prime number $p = p(E,x)$ such that for any prime number $\ell$ we have
\[
\ell \mid \#E(\mbf_p) \Longrightarrow \ell > x.
\]
We remark that one could likely prove something stronger by employing the appropriate tools.  In the interest of brevity and simplicity, we content ourselves with Theorem A\ref{maintheorem}.

We will begin by describing precisely the constant $\mf{S}_E$, from which it will be evident that the converse of Theorem A\ref{maintheorem} holds, i.e. for any elliptic curve $E$ over $\mbq$, one has
\begin{equation} \label{converseoftheorem}
\mf{S}_E = 0 \, \Longrightarrow \, \{ c_E(p) : p \nmid N_E \} \text{ is bounded.}
\end{equation}
We will then prove Theorem A\ref{maintheorem}.  Finally, we will discuss the issue of exactly when one has $\mf{S}_E > 0$ and give an example of an elliptic curve $E$ over $\mbq$ for which $\mf{S}_E = 0$ (and for which $\{ c_E(p) : p \nmid N_E \}$ is bounded, thus illustrating \eqref{converseoftheorem}).  Throughout, $\ell$ and $p$ will always denote prime numbers.

\section{The heuristic of Conjecture A\ref{koblitzconjecture} and the constant $\mf{S}_E$} \label{theconstant}

The heuristic leading to Conjecture A\ref{koblitzconjecture} is analogous to the one which leads to the classical twin prime conjecture (see \cite{koblitz} and \cite{zywina} for more details), and changes slightly depending on whether or not $E$ has complex multiplication (CM).  As usual, for $p \nmid N_E$, define the integer $a_E(p)$ by the formula
\begin{equation} \label{numberofpoints}
\# E(\mbf_p) =: p + 1 - a_E(p).
\end{equation}
By a theorem due originally to Hasse, we have that $|a_E(p)| \leq 2\sqrt{p}$, and so the size of $\#E(\mbf_p)$ is near the size of $p$.  Thus, regarding $p$ and $\#E(\mbf_p)$ as two independently chosen random positive integers of size $x$, the ``probability'' that they are both prime should satisfy
\begin{equation} \label{naiveheuristic}
\mc{P}(p \text{ is prime and } \#E(\mbf_p) \text{ is prime}) \approx \frac{1}{(\log x)^2},
\end{equation}
by the prime number theorem.  However, this prediction fails to take into account arithmetic information about the reductions of $p$ and $\#E(\mbf_p)$ modulo positive integers.
In order to describe how one corrects the situation, we begin by recalling the division fields attached to $E$ and  Chebotarev density theorem.

\subsection{The division fields $\mbq(E[n])$ of $E$}

For each positive integer $n \geq 1$ denote by
\[
E[n] := \{ P \in E(\ol{\mbq}) : [n](P) =  \mc{O}_E \}
\]
the $n$-torsion of $E$ and by $\mbq(E[n])$ the $n$-th division field of $E$, i.e. the field generated by the $x$ and $y$ coordinates of each $P \in E[n]$.  The field $\mbq(E[n])$ is a Galois extension of $\mbq$, and by fixing a $\mbz/n\mbz$-basis of $E[n]$, we may (and henceforth will) view $\gal(\mbq(E[n])/\mbq)$ as a subgroup of $GL_2(\mbz/n\mbz)$:
\[
\gal(\mbq(E[n])/\mbq) \subseteq GL_2(\mbz/n\mbz).
\]

The following proposition, which relates $p$ and $a_E(p)$ with $\mbq(E[n])$ is well-known.  In its statement $\gs_{\mbq(E[n])/\mbq}(p) \subset \gal(\mbq(E[n])/\mbq) \subseteq GL_2(\mbz/n\mbz)$ denotes the conjugacy class of a Frobenius automorphism at $p$, which we view as a subset of $GL_2(\mbz/n\mbz)$.
\begin{aproposition} \label{wellknown}
For any positive integer $n$ and any prime $p$ of good reduction for $E$ which does not divide $n$, $p$ is unramified in $\mbq(E[n])$.  Furthermore,
\begin{equation*} \label{trace}
\tr ( \gs_{\mbq(E[n])/\mbq}(p) ) \equiv a_E(p) \mod n
\end{equation*}
and
\begin{equation*} \label{det}
\det ( \gs_{\mbq(E[n])/\mbq}(p) ) \equiv p \mod n.
\end{equation*}
\end{aproposition}

\subsection{The Chebotarev density theorem}

Recall the Chebotarev density theorem \cite{chebotarev}.  Let $L/F$ be a (finite) Galois extension of number fields and $\mc{C} \subseteq \gal(L/F)$ any subset which is stable by $\gal(L/F)$-conjugation.  Denote by $\Sigma_F$ the set of prime ideals of $F$ and
\[
\Sigma_F(x) := \{ \mf{p} \in \Sigma_F : N_{F/\mbq}(\mf{p}) \leq x \}.
\]
For each prime ideal $\mf{p} \in \Sigma_F$ which is unramified in $L$, let $\gs_{L/F}(\mf{p}) \subseteq \gal(L/F)$ denote the conjugacy class of the Frobenius element attached to any prime $\mf{P}$ of $L$ lying over $\mf{p}$. 
\begin{atheorem} \label{chebotarev}
(The Chebotarev density theorem) We have
\[
\lim_{x \rightarrow \infty} \frac{\#\{\mf{p} \in \Sigma_F(x) : \mf{p} \text{ unramified in $L$ and } \gs_{L/F}(\mf{p}) \subseteq \mc{C}\}}{\# \Sigma_F(x)} = \frac{\# \mc{C}}{ \# \gal(L/F)}.
\]
\end{atheorem}
In probabilistic terms, Theorem A\ref{chebotarev} says that the probability that a randomly selected prime ideal $\mf{p}$ satisfies $\gs_{L/K}(\mf{p}) \subseteq \mc{C}$ is $\frac{\# \mc{C}}{ \# \gal(L/F)}$.

\subsection{Correcting the naive heuristic \eqref{naiveheuristic}}

For any positive integer $n$ and subgroup $G \leq GL_2(\mbz/n\mbz)$, define the subset $\Omega_n(G) \subseteq G$ by
\begin{equation} \label{defofomega}
\Omega_n(G) := \{ g \in G : \det(g) + 1 - \tr(g) \notin (\mbz/n\mbz)^* \}.
\end{equation}
The probability that a large randomly chosen integer is coprime to $n$ is $\begin{displaystyle} \frac{\#(\mbz/n\mbz)^*}{\#(\mbz/n\mbz)} \end{displaystyle}$.  On the other hand, by \eqref{numberofpoints}, Proposition A\ref{wellknown} and Theorem A\ref{chebotarev}, the probability that $\#E(\mbf_p)$ is coprime with $n$ is
\[
\frac{\#(\gal(\mbq(E[n])/\mbq) - \Omega_n(\gal(\mbq(E[n])/\mbq)))}{ \#(\gal(\mbq(E[n])/\mbq))}.
\]
Thus, it is natural to multiply \eqref{naiveheuristic} by the correction factor
\begin{equation} \label{correctionfactor}
\frac{\begin{displaystyle} \frac{\#(\gal(\mbq(E[n])/\mbq) - \Omega_n(\gal(\mbq(E[n])/\mbq)))}{ \#(\gal(\mbq(E[n])/\mbq))} \end{displaystyle}}{\begin{displaystyle} \frac{\#(\mbz/n\mbz)^*}{\#(\mbz/n\mbz)}\end{displaystyle}}.
\end{equation}
Noting that
\begin{equation*} \label{pullback}
\Omega_n(\gal(\mbq(E[n])/\mbq)) = \pi^{-1}\left( \Omega_{\gd(n)} \left( \gal(\mbq(E[\gd(n)])/\mbq) \right) \right),
\end{equation*}
where 
$\gd(n) := \prod_{\ell \mid n} \ell$
denotes the square-free kernel of $n$ and  $\pi : GL_2(\mbz/n\mbz) \twoheadrightarrow GL_2(\mbz/\gd(n)\mbz)$ denotes the canonical projection, we see that \eqref{correctionfactor} only depends on $\gd(n)$, and so it suffices to consider square-free $n$.  Defining
\begin{equation} \label{squarefreen}
n = n(z) := \prod_{\ell \leq z} \ell
\end{equation}
to be the square-free number supported on primes $\ell \leq z$, we multiply \eqref{naiveheuristic} by \eqref{correctionfactor} and take the limit as $z \rightarrow \infty$, arriving at Conjecture \ref{koblitzconjecture} with
\begin{equation} \label{constantse}
\mf{S}_E := \lim_{z \rightarrow \infty} \frac{\left(1 - \begin{displaystyle} \frac{\# \Omega_{n(z)}(\gal(\mbq(E[n(z)])/\mbq))}{\# \gal(\mbq(E[n(z)])/\mbq)} \end{displaystyle} \right)}{\begin{displaystyle} \prod_{\ell \mid n(z)} (1 - 1/\ell) \end{displaystyle}}.
\end{equation}

Our next proposition describes $\mf{S}_E$ in more detail.  In particular, it implies that the limit in \eqref{constantse} converges to a finite positive limit, provided it is non-zero for each fixed $z \geq 2$.
\begin{aproposition}
Let $E$ be an elliptic curve over $\mbq$ and let $\mf{S}_E$ be defined by \eqref{constantse}.  There exists a positive square-free integer $n_E \geq 1$ and a real number $\gl_E > 0$ so that
\[
\mf{S}_E =
\frac{\left(1 - \begin{displaystyle} \frac{\# \Omega_{n_E}(\gal(\mbq(E[n_E])/\mbq))}{\# \gal(\mbq(E[n_E])/\mbq)} \end{displaystyle} \right)}{\begin{displaystyle} \prod_{\ell \mid n_E} (1 - 1/\ell) \end{displaystyle}}\cdot \gl_E.
\]
\end{aproposition} 
\begin{proof}
In the CM case, this follows from \cite[Corollaire, p. 302]{serre} and in the non-CM case from \cite[(2), p. 260]{serre}.  For more details, see \cite{zywina}.
\end{proof}
Although it won't be necessary in what follows, we remark that
\[
\gl_E = 
\begin{cases}
	\begin{displaystyle} \frac{1}{2} \cdot \prod_{\ell \nmid n_E} \left( 1 - \chi(\ell) \frac{\ell^2 - \ell - 1}{(\ell - \chi(\ell))(\ell-1)^2} \right) \end{displaystyle} & \text{if $E$ has CM by $K$}, \\
	\begin{displaystyle} \prod_{\ell \nmid n_E} \left( 1 - \frac{\ell^2 - \ell - 1}{(\ell - 1)^3(\ell + 1)} \right) \end{displaystyle} & \text{if $E$ has no CM},
\end{cases}
\]
where in the CM case, $\chi(\ell) \in \{ \pm 1\}$ denotes the Kronecker symbol giving the splitting of $\ell$ in the imaginary quadratic field $K$.
\begin{acorollary} \label{betterce}
We have
\begin{equation*} 
\mf{S}_E = 0 \; \Longleftrightarrow \; \left( \exists \textrm{ square-free }n_0, \; \Omega_{n_0}(\gal(\mbq(E[n_0])/\mbq)) = \gal(\mbq(E[n_0])/\mbq)\right).
\end{equation*}
\end{acorollary}
In particular, if $\mf{S}_E = 0$, then by \eqref{numberofpoints}, Proposition A\ref{wellknown} and Theorem A\ref{chebotarev}, we have
\begin{equation} \label{cantbeprime}
p \nmid n_0 \cdot N_E \; \Longrightarrow \; \gcd(\#E(\mbf_p), n_0) > 1.
\end{equation}
Since this in turn causes $\{ c_E(p) : p \nmid N_E \}$ to be bounded, we have verified \eqref{converseoftheorem}.

\section{Proof of Theorem A\ref{maintheorem}} \label{proofoftheorem}

To prove Theorem A\ref{maintheorem}, we will apply Theorem A\ref{chebotarev} with $F = \mbq$, $L = \mbq(E[n])$, and 
$$\mc{C} = \left( \gal(\mbq(E[n])/\mbq) - \Omega_n(\gal(\mbq(E[n])/\mbq)) \right),$$ 
with $\Omega_n(G)$ as in \eqref{defofomega} and $n = n(z)$ as in \eqref{squarefreen}.  Fix any prime $p > z$ which doesn't divide $N_E$.  By Proposition A\ref{wellknown}, $p$ is unramified in $\mbq(E[n(z)])$ 
and furthermore we have the following equivalence:
\begin{equation} \label{mainequivalence}
\left( \forall \ell \leq z, \, \ell \nmid \#E(\mbf_p) \right) \; \Longleftrightarrow \; 
\gs_{\mbq(E[n(z)])/\mbq}(p) \nsubseteq \Omega_{n(z)}(\gal(\mbq(E[n(z)])/\mbq)).
\end{equation}

Now consider the Chebotarev factor
\[
\mc{D}_z := \frac{\#(\gal(\mbq(E[n(z)])/\mbq) - \Omega_{n(z)}(\gal(\mbq(E[n(z)])/\mbq)))}{ \#(\gal(\mbq(E[n(z)])/\mbq))}.
\]
By Corollary A\ref{betterce}, we see that
\[
\mf{S}_E > 0 \; \Longrightarrow \; \mc{D}_z > 0.
\]
Thus, provided $\mf{S}_E > 0$, Theorem A\ref{chebotarev} implies the existence of a prime number $p_1 = p_1(E,z)$ for which
\[
\gs_{\mbq(E[n(z)])/\mbq}(p_1) \nsubseteq \Omega_{n(z)}(\gal(\mbq(E[n(z)])/\mbq)). 
\]
By \eqref{mainequivalence}, we see that for each $\ell \leq z$, $\ell$ does not divide $\#E(\mbf_{p_1})$, and so $c_E(p_1) \geq z$.  Since $z$ was arbitrary, Theorem A\ref{maintheorem} follows.

\section{The positivity of $\mf{S}_E$} \label{positivity}

It is now natural to ask:  under what conditions is the constant $\mf{S}_E$ positive?
Because of the Weil Pairing (see \cite{silverman}, for example), for any level $n$, we have that the determinant map restricts to a surjection
\[
\det\colon \gal(\mbq(E[n(z)])/\mbq) \twoheadrightarrow (\mbz/n(z)\mbz)^*.
\]
By Corollary A\ref{betterce}, we are thus led to ask the following question.
\begin{aquestion} \label{firstquestion}
Let $n \geq 1$ be a positive square-free integer, and let $G \leq GL_2(\mbz/n\mbz)$ be a subgroup for which the determinant map restricts to a surjection: 
\[
\det \colon G \twoheadrightarrow (\mbz/n\mbz)^*.
\]
Under which circumstances do we have $\Omega_n(G) = G$?
\end{aquestion}
It is clear from the definitions that, for any $\ell$ dividing $n$ we have
\[
\Omega_\ell(G \!\!\!\mod \ell) = G \!\!\!\mod \ell \quad \Longrightarrow \quad \Omega_n(G) = G.  
\]
We join Serre \cite[I-2]{serre} in leaving the following exercise up to the reader.
\begin{aexercise}
Prove that, for any subgroup $G_\ell \leq GL_2(\mbz/\ell\mbz)$, $\Omega(G_\ell) = G_\ell$ if and only if either
\begin{equation*} 
G_\ell \subseteq \left\{ \begin{pmatrix} 1 & * \\ 0 & * \end{pmatrix} \right\} \quad\quad \text{ or } \quad\quad G_\ell \subseteq \left\{ \begin{pmatrix} * & * \\ 0 & 1 \end{pmatrix} \right\}.
\end{equation*}
\end{aexercise}
Furthermore, $\gal(\mbq(E[\ell])/\mbq) = G_\ell$ as above if and only if $E$ is isogenous over $\mbq$ to some elliptic curve $E'$ over $\mbq$ satisfying $E'[\ell](\mbq) \neq \{ \mc{O}_{E'} \}$ (in the first case, $E'$ is simply $E$).  We record this as
\begin{aremark} \label{localobstruction}
If $E$ is $\mbq$-isogenous to some elliptic curve $E'$ over $\mbq$ for which $E'(\mbq)_{\text{tors}} \neq \{ \mc{O}_{E'} \}$, then $\mf{S}_E = 0$.
\end{aremark}

It is tempting to expect (as Koblitz did) that the converse of Remark A\ref{localobstruction} also holds, but the following example shows that this is not the case.  Let $\ell \neq 2$ be any prime and consider the subgroup $G \leq GL_2(\mbz/2\mbz) \times GL_2(\mbz/\ell\mbz)$ defined by
\begin{equation} \label{defofG}
G
=  \{ (g_2,g_\ell) \in GL_2(\mbz/2\mbz) \times G_1(\mbz/\ell\mbz) : \chi_2(g_2) = \chi_\ell(g_\ell) \},
\end{equation}
where
\begin{equation} \label{defofG1}
G_1(\mbz/\ell\mbz) := \left\{ \begin{pmatrix} \pm 1 & * \\ 0 & * \end{pmatrix} \right\} \leq GL_2(\mbz/\ell\mbz)
\end{equation}
and the characters $\chi_2$ and $\chi_\ell$ are defined by
\begin{equation*} 
\chi_2 : GL_2(\mbz/2\mbz) \longrightarrow GL_2(\mbz/2\mbz) / [GL_2(\mbz/2\mbz),GL_2(\mbz/2\mbz)] \simeq \{ \pm 1 \}
\end{equation*}
and
\begin{equation} \label{defofchiell}
\chi_\ell \left( \begin{pmatrix} \pm 1 & * \\ 0 & * \end{pmatrix} \right) = \pm 1.
\end{equation}
Notice that, even though 
\[
\Omega_2(G \!\!\!\mod 2) \subsetneq G \!\!\!\mod 2 \quad \text{ and } \quad \Omega_\ell(G \!\!\!\mod \ell) \subsetneq G \!\!\!\mod \ell,
\]
we have $\Omega_{2\ell}(G) = G$.  Provided we can find an elliptic curve $E$ over $\mbq$ with $\gal(\mbq(E[2\ell])/\mbq) \leq G$, then $\# E(\mbf_p)$ will only be prime finitely often because whenever it is not divisible by $2$, it must be divisible by $\ell$, and vice versa.


\subsection{A counterexample to \eqref{positivityofse}}

\begin{aproposition} \label{counterexampleprop}
Let $E$ be the elliptic curve defined by the Weierstrass equation
\[
y^2 = x^3 + 75x + 125.
\]
For any elliptic curve $E'$ over $\mbq$ which is $\mbq$-isogenous to $E$, 
one has $E'(\mbq)_{\text{tors}} = \{ \mc{O}_{E'} \}$.  Nevertheless, $\mf{S}_E = 0$.  
Furthermore, the Mordell--Weil group attached to $E$ is infinite:
\[
\#E(\mbq) = \infty.
\]
\end{aproposition}
\begin{proof}
Since $N_E = 2^2 \cdot 3^3 \cdot 5^2$, we see that $E$ has good reduction away from 
$p \in \{2,3,5\}$.  One calculates that $\#E(\mbf_7) = 4$ and $\#E(\mbf_{17}) = 21$, 
from which it follows that, for any $E'$ over $\mbq$ which is $\mbq$-isogenous to $E$, 
we have $E'(\mbq)_{\text{tors}} = \{ \mc{O}_{E'} \}$.  On the other hand, recall that 
$\mbq(E[2]) = \mbq(\text{the roots of } x^3+75x+125)$, so that 
$\sqrt{\gD_E} = 2^2 \cdot 3 \cdot 5^3\sqrt{-15} \in \mbq(E[2])$.  
Also, the point $(-5,5\sqrt{-15}) \in E[3](\mbq(\sqrt{-15}))$ shows that 
\[
\mbq(\sqrt{\gD_E}) = \mbq(\sqrt{-15}) \subseteq \mbq(E[2]) \bigcap \mbq(E[3]).
\]
It follows that, taking $\ell = 3$ in \eqref{defofG}, we have 
$\gal(\mbq(E[6])/\mbq) \leq G$, where $\chi_2$ and $\chi_\ell$ correspond to the 
restriction map
\[
\gal(\mbq(E[6])/\mbq) \longrightarrow \gal(\mbq(\sqrt{\gD_E})/\mbq) \simeq \{ \pm 1 \}.
\]
Taking $n_0 = 6$ in Corollary A\ref{betterce}, we see that $\mf{S}_E = 0$.

Finally, the point $(5,25) \in E(\mbq)$ is of infinite order, and so $\#E(\mbq) = \infty$, 
as claimed.
\end{proof}

Furthermore, one can readily verify \eqref{cantbeprime} with $n_0 = 6$ and $E$ as in 
Proposition A\ref{counterexampleprop}, as follows.
For any rational prime $p \geq 7$ and choice of Frobenius automorphism 
$\gs_6(p) \in \gs_{\mbq(E[6])/\mbq}(p)$, we have that
\[
\gs_6(p)(\sqrt{\gD_E}) = \sqrt{\gD_E} \Rightarrow \gs_{\mbq(E[3])/\mbq}(p) 
\subseteq \Omega_3(\gal(\mbq(E[3])/\mbq)) \Rightarrow 3 \mid \#E(\mbf_p)
\]
and
\[
\gs_6(p)(\sqrt{\gD_E}) = -\sqrt{\gD_E} \Rightarrow \gs_{\mbq(E[2])/\mbq}(p) 
\subseteq \Omega_2(\gal(\mbq(E[2])/\mbq)) \Rightarrow 2 \mid \#E(\mbf_p).
\]
Since $\sqrt{\gD_E} = 2^2\cdot 3 \cdot 5^3 \sqrt{-15}$, it follows that for 
$p \nmid 30$, we have
\[
\left( \frac{-15}{p} \right) = 1 \Rightarrow 3 \mid \#E(\mbf_p)
\]
and
\[
\left( \frac{-15}{p} \right) = -1 \Rightarrow 2 \mid \#E(\mbf_p). 
\]
This verifies \eqref{cantbeprime} and shows that
\[
\{ c_E(p) : p \nmid N_E \} = \{ 2, 3 \}.
\]

More generally, we have
\begin{aremark} \label{nonlocalobstruction}
If $E$ is $\mbq$-isogenous to some elliptic curve $E'$ over $\mbq$ for which $E'(\mbq(\sqrt{\gD_{E'}}))_{\text{tors}} \neq \{ \mc{O}_{E'} \}$, then $\mf{S}_E = 0$.
\end{aremark}

Have we covered all possible cases where $\mf{S}_E = 0$?  We will now give an example of a subgroup $G \leq GL_2(\mbz/3\ell\mbz)$ satisfying $\Omega_{3\ell}(G) = G$, where $\ell \geq 5$ is some prime.  Let
\[
\mc{N}_{3} := \left\{ \pm \begin{pmatrix} 1 & 0 \\ 0 & 1 \end{pmatrix}, \pm \begin{pmatrix}  0 & -1 \\ 1 & 0 \end{pmatrix} \right\}  \sqcup \left\{ \pm \begin{pmatrix} 1 & 0 \\ 0 & -1 \end{pmatrix}, \pm \begin{pmatrix}  0 & 1 \\ 1 & 0 \end{pmatrix} \right\} \leq GL_2(\mbz/3\mbz),
\]
and define
\[
G := \{ (g_3,g_\ell) \in \mc{N}_3 \times G_1(\mbz/\ell\mbz) : \det g_3 = \chi_\ell(g_\ell) \},
\]
where $G_1(\mbz/\ell\mbz)$ and $\chi_\ell$ are as in \eqref{defofG1} and \eqref{defofchiell}, respectively, and we are regarding $\det(g_3) \in \mbf_3^* = \{\pm 1\}$.  As before, we have 
\[
\Omega_3(G \!\!\!\mod 3) \subsetneq G \!\!\!\mod 3 \quad \text{ and } \quad \Omega_\ell(G \mod \!\!\! \ell) \subsetneq G \mod \!\!\! \ell,
\]
but $\Omega_{3\ell}(G) = G$.  Perhaps there may also be an elliptic curve $E$ over $\mbq$ with $\gal(\mbq(E[3\ell])/\mbq) \leq G$, though we haven't explicitly exhibited one.

\subsection{Serre curves}

A \textbf{Serre curve} is an elliptic curve $E$ over $\mbq$ for which
\[
\forall n \geq 1, \quad [GL_2(\mbz/n\mbz) : G_E(n)] \leq 2.
\]
(Intuitively, a Serre curve is an elliptic curve for which $\gal(\mbq(E[n])/\mbq)$ is ``as large as possible'' for each $n \geq 1$.)  We remark that, as shown in \cite[Proposition $4.2$]{zywina}, we have 
\[
E \textrm{ is a Serre curve } \Longrightarrow \; \mf{S}_E > 0.
\]
When ordered according to naive height, almost all elliptic curves are Serre curves (see \cite{jonessc}).  Thus, for a ``typical'' elliptic curve $E$ over $\mbq$ one has $\mf{S}_E > 0$.

\section{Concluding remarks} \label{concludingremarks}

As mentioned in the introduction, one can likely prove stronger forms of Theorem A\ref{maintheorem}.  For instance, one could probably use an effective version of the Chebotarev density theorem to obtain a quantitative upper bound for the smallest prime $p$ for which $c_E(p) > x$.

Since we have not completely resolved it, we record here
\begin{aquestion} \label{finalquestion}
Under what conditions do we have $\mf{S}_E > 0$?  
\end{aquestion}
The examples discussed in Section \ref{positivity} seem to indicate that this question is more delicate than it first may seem.
  Conjecture A\ref{koblitzconjecture} has also been generalized to the context where $E$ is defined over a general number field $K$ (see \cite{zywina}), in which case the answer to Question A\ref{finalquestion} may become even more delicate.

\def\refname{References to the appendix}

\end{document}